\documentclass{article}
\usepackage[ruled,vlined, linesnumbered]{algorithm2e}
\usepackage{amsmath,url,cases,supertabular,longtable}
\usepackage{amssymb,mathtools,multirow}
\usepackage{enumerate,makecell}
\usepackage{amsfonts}

\usepackage{booktabs}
\usepackage{graphicx}
\usepackage{rotating}
\usepackage{tikz}
\usepackage{collcell}
\usepackage{xstring}
\usepackage{array}
\usepackage{balance}
\usepackage{amsbsy} 
\usepackage{comment}
\usepackage{balance} 
\usepackage{amsfonts}
\usepackage{multirow}

\usepackage{url}
\usepackage{graphics}
\usepackage[linesnumbered,ruled,vlined]{algorithm2e}
\usepackage{algorithmicx}
\usepackage{algpseudocode}
\usepackage{asymptote}
\usepackage{amsmath}
\usepackage[colorlinks]{hyperref}
\usepackage{pdflscape}
\usepackage{color, colortbl}
\usepackage{enumitem}

\newcommand{\bit}{\begin{itemize}}
\newcommand{\eit}{\end{itemize}}
\newcommand{\be}{\begin{equation}}
\newcommand{\ee}{\end{equation}}
\newcommand{\bea}{\begin{eqnarray}}
\newcommand{\eea}{\end{eqnarray}}

\newcommand{\argmax}{\mathop{\mathrm{arg\,max}}}

\newcommand{\myvec}[1]{\mathbf{#1}}
\newcommand\alg[1]{\texttt{#1}}
\newcommand\myeq{\mathrel{\stackrel{\makebox[0pt]{\mbox{\normalfont\tiny def}}}{=}}}
\newcommand{\ie}{i.e.}
\newcommand{\eg}{e.g.}

\newcommand*\Bell{\ensuremath{\boldsymbol\ell}}
\newcommand{\hv}{$I^{-}_{H}$}
\newcommand{\epsind}{$I^{1}_{\epsilon+}$}
\newcommand{\IUnaryAddInd}{$I^{1}_{\epsilon+}(\mathcal{Y}^t_*)$}

\usepackage{amsthm}
\newtheorem{remark}{Remark}

\newtheorem{thm}{Theorem}
\newtheorem{lem}{Lemma}

\newtheorem{definition}{Definition}

\SetKwFunction{ND}{ND}

\definecolor{NavyBlue}{RGB}{35,35,142}
\definecolor{RawSienna}{RGB}{199,97,20}
\hypersetup{
	colorlinks,%
	citecolor=NavyBlue,%
	filecolor=NavyBlue,%
	linkcolor=RawSienna,%
	urlcolor=NavyBlue
}
\makeatletter
\def\LT@makecaption#1#2#3{%
 \LT@mcol\LT@cols c{\hbox to\z@{\hss\parbox[c]\LTcapwidth{%
     \footnotesize
     \setlength{\parindent}{1.5pc}
   \hbox to\hsize{\hfil #1{\normalfont \scshape #2 }  \hfil}
    \setbox\@tempboxa\hbox{{\normalfont\itshape #3}}
   \ifdim\wd\@tempboxa>\hsize
    {\hspace{0pt}\normalfont\itshape #3}\par 
   \else
     \hbox to\hsize{\hfil\box\@tempboxa\hfil}%
   \fi
    \vskip\belowcaptionskip
  }%
 \hss
 } } }

\makeatother

\title{ Multi-Objective Simultaneous Optimistic Optimization
}
\author{Abdullah Al-Dujaili\thanks{School of Computer Engineering, Nanyang Technological University, Singapore 639798 (aldujail001@e.ntu.edu.sg).  Support for this author was provided by ATMRI:2014-R8, Singapore.
}\and S. Suresh\thanks{School of Computer Engineering, Nanyang Technological University, Singapore 639798 (ssundaram@ntu.edu.sg).  Support for this author was provided by ATMRI:2014-R8, Singapore.
}}
\date{\today}
\begin{document}
\maketitle

\begin{abstract}
Optimistic methods have been applied with success to single-objective optimization. Here, we attempt to bridge the gap between optimistic methods and multi-objective optimization. In particular, this paper is
concerned with solving black-box multi-objective  problems given a finite number of function evaluations and proposes an optimistic approach, which we refer to as the Multi-Objective Simultaneous Optimistic Optimization (\alg{MO-SOO}).
Popularized by multi-armed bandits, \alg{MO-SOO} follows the optimism in the face of uncertainty principle to recognize Pareto optimal solutions, by combining several multi-armed bandits in a hierarchical structure over the feasible decision space of a multi-objective problem.
Based on three assumptions about the objective functions smoothness and hierarchical partitioning, the algorithm finite-time and asymptotic convergence behaviors are analyzed. The finite-time analysis establishes an upper bound on the Pareto-compliant unary  additive epsilon indicator characterized by the objectives smoothness as well as the structure of the Pareto front with respect to its extrema. On the other hand, the asymptotic analysis indicates the consistency property of \alg{MO-SOO}. Moreover, we validate the theoretical provable performance of the algorithm on a set of synthetic problems.
Finally, three-hundred bi-objective benchmark problems from the literature are used to substantiate the performance of the optimistic approach and compare it with three state-of-the-art stochastic algorithms, namely  \alg{MOEA/D}, \alg{MO-CMA-ES}, and \alg{SMS-EMOA} in terms of two Pareto-compliant quality indicators. Besides sound theoretical properties, \alg{MO-SOO} shows a performance on a par with the top performing stochastic algorithm, viz. \alg{SMS-EMOA}.
\end{abstract}

\pagestyle{myheadings}
\thispagestyle{plain}
\markboth{\sc{Multi-Objective Simultaneous Optimistic Optimization}}{\sc{Abdullah Al-Dujaili, S. Suresh}}

\section{Introduction}
\label{sec:intro}

Many real-world  application and decision problems involve optimizing two or more objectives at the same time~(see, \eg,~\cite{depraetere2010iterative,aguilar2010surgical}). These problems are often referred to as Multi-Objective Optimization (\textsc{MOO}). In the general case, MOO problems are hard because the objective functions are often conflictual, and it is difficult to design strategies that are optimal for all objectives simultaneously. Furthermore, with conflicting objectives, there  does not exist a single optimal solution but a set of incomparable optimal solutions: each is inferior to the other in some objectives and superior in other objectives. This induces a partial order on the set of feasible solutions to an MOO problem. The set of optimal feasible solutions according to this partial order is referred to as the \emph{Pareto optimal set} and its corresponding image in the objective space is commonly named as the \emph{Pareto front} of the problem.  The task of MOO algorithms therefore becomes finding the Pareto front or producing a good approximation of it (referred to as an \emph{approximation set} of the problem).

Generally, certain assumptions are made about the objective functions being optimized (e.g., its continuity or differentiability). However, these assumptions are not necessarily satisfied by real-world problems. Sometimes, the only information available about the objective functions are their point-wise evaluations: computing their derivatives or other measures are either expensive, unreliable, or even impossible.  Such problems are called \textit{black-box} multi-objective optimization problems and appear very often in real-world settings~\cite{hoffmann2006derivative}. In this paper, we study the problem of black-box MOO given a finite number of objective functions evaluations (often referred to as the evaluation budget).

Conventionally, solving a multi-objective optimization problem follows one of two principles, namely \emph{preference-based} and \emph{ideal} principles \cite{deb2001multi}. Following the preference-based principle, the MOO problem is transformed into a single-objective optimization problem (through an aggregation/scalarization function that exploits a priori information), which then can be solved using one of many available single-objective optimizers~\cite{26_p,miettinen1999nonlinear,mannor2004geometric}. While preference-based algorithms converge to a single solution in each run, ideal-based algorithms search for a set of solutions at once. One example in this approach is evolutionary multi-objective algorithms~\cite{deb2002fast,spea2} in which a population of solutions evolves, following a crude analogy with Darwinian evolution, towards better solutions. Recently, there has been a growing interest of formulating multi-objective problems within the framework of reinforcement learning (see, for instance,~\cite{gabor1998multi,natarajan2005dynamic,barrett2008learning,liao2010multi,roijers2013survey}).

Among the several lessons learned from the aforementioned MOO solvers over the past decades is that, in order to generate a dense and good approximation set, one must maintain the set diversity. Furthermore, one must not discard inferior solutions too easily, as some of them may pave the way towards rarely-visited regions of the Pareto front \cite{loshchilov:phd}. In other words, the exploration-vs.-exploitation trade-off in search for the Pareto optimal set should be thought carefully about, at the \emph{algorithmic design level}. With this regard, in this paper, we are motivated to address the problem of multi-objective optimization within the framework of \emph{optimistic sequential decision-making methods}. \ie, methods that implement the \textit{optimism in the face of uncertainty} principle. Such principle finds its foundations in the machine learning field addressing the exploration-vs.-exploitation dilemma, known as the \textit{multi-armed bandit problem} (introduced independently by~\cite{thompson1933likelihood} and~\cite{robbins1952some}).

Within the context of single-objective optimization, optimistic sequential decision-making approaches formulate the complex problem of global optimization over the decision space~$\mathcal{X}$ as a hierarchy of simple bandit problems over subspaces of~$\mathcal{X}$ and look for  the optimal solution through~$\mathcal{X}$-partitioning search trees: each leaf corresponds to a subspace of~$\mathcal{X}$, with the root corresponding to~$\mathcal{X}$ and nodes at depth~$h\in\mathbb{N}_0$ represent a partition of~$\mathcal{X}$ at scale~$h$. At step $t$, such algorithms optimistically expand a leaf node (\ie, partition the corresponding subspace) that may contain the optimum. In other words, optimistic algorithms consider partitions of the search space at multiple scales in search for the optimal solution~\cite{munos2011optimistic,bubeck2011x,wang2014bayesian}. Recently, the optimistic optimization algorithm, {Naive Multi-scale Search Optimization}~\cite{ash-nmso-15}, has been shown to be a viable alternative to solve black-box optimization problems---see the results of the Black-Box Optimization Competition (\textsc{BBComp}) within the Genetic and Evolutionary
Computation Conference (GECCO'2015)~\cite{bbcomp15}.

On the other hand, two observations can be made about optimistic methods within the context of multi-objective optimization.  First, there has been very little/limited yet slowly growing research reported on optimistic methods for multi-objective optimization. For instance, the focus of multi-objective multi-armed bandit problems has been distinctly on a discrete set of arms~\cite{drugan2013designing}, or solving a subproblem (\eg, selecting a genetic operator in evolutionary multi-objective algorithms \cite{LiFKZ14}). Second, the algorithmic development and validation have been dominantly empirical~(see, for instance, \cite{van2014multi,wang:hal-00758379,vamplew2011empirical}).

Being one of the simplest single-objective optimistic methods with a theoretically provable performance, this paper is inspired by  the Simultaneous Optimistic Optimization (\alg{SOO})~\cite{munos2011optimistic} to develop an optimistic algorithm for multi-objective problems. We refer to this algorithm as the Multi-Objective Simultaneous Optimistic Optimization (\alg{MO-SOO}). In order to find a good approximation set of the Pareto front, \alg{MO-SOO} employs---similar to optimistic methods ---hierarchical bandits over the decision space. Represented by a divide-and-conquer tree structure, the hierarchical bandits are realized by partitioning the decision space over multiple scales. At each step, \alg{MO-SOO} expands leaf nodes (partitions the corresponding subspaces) that may optimistically contain \emph{Pareto optimal} solutions. Based on three assumptions about the function smoothness and partitioning strategy, we analyze the finite-time and asymptotic convergence behaviors of \alg{MO-SOO}. The finite-time study is based on quantifying how much exploration is required to achieve near-optimal objective-wise solutions. As a result, we are able to upper bound the loss of the obtained solutions with respect to the objective-wise optimal solutions. Using this objective-wise loss bound, an upper bound on the Pareto-compliant unary additive epsilon indicator~\cite{zitzler2003performance} is established as a function of the number of iterations. The bound is characterized by the objectives smoothness as well as the structure of the Pareto front with respect to its extrema. First time in the literature, a deterministic upper bound on a Pareto-compliant indicator is presented for a solver of continuous MOO problems. However, the presented bound holds down to a problem-dependent constant. Furthermore, the systematic sampling nature of the decision space in \alg{MO-SOO} helps in analyzing the asymptotic behavior, which indicates its consistency, viz. optimality in the limit. Using symbolic maths, the theoretical provable performance of the algorithm has been validated on a synthetic problem.

Complementing the theoretical results, an empirical validation study has been conducted using 300 bi-objective benchmark problems from the literature~\cite{brockhoff:hal-01146741}. The test suite considers problems with various objective functions categories reflecting real-world scenarios such as separability and multi-modality. It can also be used to validate the algorithms scalability with the decision space dimension.

Furthermore, \alg{MO-SOO} has been compared with 3 state-of-the-art stochastic algorithms, namely \alg{MOEA/D}~\cite{zhang2007moea}, \alg{MO-CMA-ES}~\cite{voss2010improved} , and \alg{SMS-EMOA}~\cite{beume2007sms} in terms of two Pareto-compliant quality indicators~\cite{knowles_tutorial_indicator}: the hypervolume (\hv) and the unary additive $\epsilon$-indicator (\epsind). The results are presented in form of data profiles, which adequately capture the convergence behavior of the algorithms over the number of function evaluations used. \alg{MO-SOO} shows a comparable performance with the top performing stochastic algorithm, viz. \alg{SMS-EMOA}.

The rest of the paper is organized as follows. Section~\ref{sec:formal_bg} discusses briefly related formal background. Section~\ref{sec:mosoo} presents the \alg{MO-SOO} algorithm and provides a worked example. Then, the algorithm's finite-time and asymptotic convergence is studied in Section~\ref{sec:conv} with supporting illustrations.  Numerical assessment of \alg{MO-SOO} is discussed in Section~\ref{sec:assessment}. Section~\ref{sec:conclusion} concludes the paper.

\section{Formal Background}
\label{sec:formal_bg}
This section introduces the main notations and terminology used in the rest of the paper. Furthermore, it provides a brief description of the multi-objective optimization problem and the optimistic approach in optimization.
\subsection{Multi-Objective Optimization}
\label{sec:bg:bbmo}

\textit{Without loss of generality}, the multi-objective minimization problem with $n$ decision variables and $m$ objectives, has the form:

\begin{equation}
	\begin{aligned}
		& {\text{minimize}}
		& & \myvec{y} = \myvec{f}(\myvec{x}) = (f_1(\myvec{x}), \ldots, f_m(\myvec{x}))\\
		& \text{where}
		& & \myvec{x} =  (x_1, \ldots, x_n)\in \mathcal{X} \\
		& & & \myvec{y} =  (y_1, \ldots, y_m)\in \mathcal{Y}
	\end{aligned}
	\label{eq:problem_def}
\end{equation}
and where $\myvec{x}$ is called the \emph{decision vector (solution)}, $\myvec{y}$ is called the \emph{objective vector},\footnote{For brevity, we sometimes omit the word \emph{objective} when referring to an objective vector. 
} $\mathcal{X}$ 
is the \emph{feasible decision space}
, and $\mathcal{Y}={\text{\huge $\times$}}_{1 \leq j\leq m} \mathcal{Y}_j$ is the corresponding \emph{objective space}, where $\mathcal{Y}_j$ is the $j$th-objective space and we write the corresponding image in the objective space for any region $\hat{\mathcal{X}}\subseteq \mathcal{X}$ as $\myvec{f}(\hat{\mathcal{X}})\subseteq\mathcal{Y}$. It is assumed that: the derivatives of the functions involved are neither symbolically nor numerically available; nevertheless, $\myvec{f}$ can be evaluated point-wise; and that evaluating it is typically expensive, requiring some computational resources (\eg, time, power, money). More specifically, the task is to best approximately solve (in a sense to be defined later) Eq.~\eqref{eq:problem_def} using a computational budget of $v$ function evaluations.

A vector $\myvec{y}^1$ is more preferable than another vector $\myvec{y}^2$, if $\myvec{y}^1$ is at least as good as $\myvec{y}^2$ in all objectives \emph{and} better with respect to at least one objective. $\myvec{y}^1$ is then said to be \emph{dominating} $\myvec{y}^2$. This notion of dominance is commonly known as \textit{Pareto dominance}~\cite{pareto-book}, which leads to a \textit{partial order} on the objective space, where we can define a Pareto optimal vector to be one that is non-dominated by any other vector in $\mathcal{Y}$.  Nevertheless,  $\myvec{y}^1$ and $\myvec{y}^2$ may be incomparable to each other, because each is inferior to the other in some objectives and superior in other objectives. Hence, there can be several Pareto optimal vectors. The following definitions put these concepts formally, in line with \cite{loshchilov:phd,zitzler2003performance}.

\begin{definition}[Pareto dominance] The vector $\myvec{y}^1$ dominates the vector $\myvec{y}^2$, that is to say, $\myvec{y}^1 \prec\myvec{y}^2$$\iff$ $y^1_j\leq y^2_j$ for all $j\in \{1,\ldots,m\}$ and $y^1_k < y^2_k$ for at least one $k\in\{1,\ldots,m\}$.
	\label{def:parteo_dominance}
\end{definition}

\begin{definition}[Strict Pareto dominance] The vector $\myvec{y}^1$ strictly dominates the vector~$\myvec{y}^2$ if $\myvec{y}^1$ is better than $\myvec{y}^2$ in all the objectives, that is to say, $\myvec{y}^1 \prec\prec\myvec{y}^2\iff$ $y^1_j< y^2_j$ for all $j\in \{1,\ldots,m\}$.
	\label{def:strict_parteo_dominance}
\end{definition}

\begin{definition}[Weak Pareto dominance] The vector $\myvec{y}^1$ weakly dominates the vector~$\myvec{y}^2$ if $\myvec{y}^1$ is not worse than $\myvec{y}^2$ in all the objectives, that is to say, $\myvec{y}^1 \preceq\myvec{y}^2\iff$ $y^1_j\leq y^2_j$ for all $j\in \{1,\ldots,m\}$.
	\label{def:weak_parteo_dominance}
\end{definition}

\begin{definition}[Pareto optimality of vectors]
	\label{def:paretoptimal}
	Let $\hat{\myvec{y}}\in \mathcal{Y}$ be a vector. $\hat{\myvec{y}}$ is Pareto optimal $\iff$ $\nexists \myvec{y} \in \mathcal{Y}$ such that  $\myvec{y}\prec \hat{\myvec{y}}$. The set
	of all Pareto optimal vectors is referred to as the Pareto front and denoted as $\mathcal{Y}^*$. The corresponding decision vectors (solutions) are referred to as the Pareto optimal solutions or the Pareto set and denoted by $\mathcal{X}^*$.
\end{definition}
In other words, the solution to the \textsc{MOO} problem \eqref{eq:problem_def} is its Pareto optimal solutions (Pareto front in the objective space). Practically, \textsc{MOO} solvers aim to identify a set of objective vectors that represent the Pareto front (or a good approximation of it). We refer to this set as the approximation set. 
\begin{definition}[Approximation set]
	Let $A \subseteq \mathcal{Y}$ be
	a set of objective vectors. $A$ is called an approximation
	set if any element of $A$ does not dominate or is
	not equal to any other objective vector in $A$. The set
	of all approximation sets is denoted as $\Omega$. Note that $\mathcal{Y}^*\in \Omega$.
	\label{def:approx_set}
\end{definition}
Furthermore, denote the \emph{ideal point (vector)} (not necessarily reachable) by $\myvec{y}^*\myeq$ $(\min_{\myvec{y}\in\mathcal{Y}^*} y_1,$
$\ldots,\min_{\myvec{y}\in\mathcal{Y}^*} y_m)$. Likewise, let us denote the (or one of the) global optimizer(s) of the $j$th objective function by $\myvec{x}^*_j$, \ie, $y^*_j=f_j(\myvec{x}^*_j)$. Note that $\myvec{x}^*_j\in \mathcal{X}^*$. On the other hand, we define the \emph{nadir point} of a region in the objective space~$\hat{\mathcal{Y}}\subseteq \mathcal{Y}$ as $\myvec{y}^{nadir}(\hat{\mathcal{Y}})\myeq(\max_{\myvec{y}\in\hat{\mathcal{Y}}} y_1,$
$\ldots,\max_{\myvec{y}\in\hat{\mathcal{Y}}} y_m)$.
\subsection{Optimistic Optimization}
\label{sec:bg:oo}
The \emph{optimism in the face of uncertainty} principle recommends following the optimal strategy with respect to the most favorable scenario among all possible scenarios that are compatible with the obtained observations about the problem at hand~\cite{MunosFTML2014}. This principle has been applied primarily within the framework of multi-armed bandit problem \cite{auer2002finite} and later was extended to many (possibly infinite) arms under a probabilistic or structural (smoothness) assumption about the arm rewards. An algorithmic instance was the Monte Carlo tree search, which witnessed an experimental success in computer GO~\cite{wang2007modifications}.

With this regard, global continuous optimization can be  modeled as a structured bandit problem where the objective value is a function of some arm parameters \cite{auer2003using,rusmevichientong2010linearly}. Based on the observations and the smoothness assumption, an optimistic strategy would compute a bound on the objective (reward) value at each solution (arm) $\myvec{x}\in \mathcal{X}$  and choose the arm with the best bound. Examples of global continuous optimization algorithms with a closely related approach are Lipschitzian optimization techniques \cite{pinter1995}.  However, this approach poses two problems: i) the computational complexity of computing the bounds over $\mathcal{X}$ at each step; ii) the restriction that the smoothness assumption puts on the objective functions that can be optimized. While the second issue can be addressed with weak, yet effective assumptions on the function smoothness, \eg, local (rather than global) smoothness; the first issue can be alleviated by transforming the problem from a many-arm bandit to a hierarchy of multi-armed bandits (often referred to as \textit{hierarchical bandits} \cite{kocsis2006bandit}). Hence, an optimistic optimization algorithm can be regarded as a tree-search divide-and-conquer algorithm that iteratively constructs finer and finer partitions of the search space $\mathcal{X}$ at multiple scales $h\in\mathbb{N}_0$. Given a scale~$h\geq 0$ and a partition factor $K \geq 2$, $\mathcal{X}$ can be partitioned into a set of $K^h$ cells/hyperrectangles/subspaces $\mathcal{X}_{h,i}$ where $0 \leq i \leq K^h-1$ such that $\cup_{i\in\{0,\ldots,K^h-1\}} \mathcal{X}_{h,i} = \mathcal{X}$. These cells are represented by nodes of a $K$-ary tree $\mathcal{T}$ (as shown in Figure~\ref{fig:hierarchical_partition}), where a node $(h,i)$ represents the cell $\mathcal{X}_{h,i}$ (the root node $(0,0)$ represents the entire search space $\mathcal{X}_{0,0}=\mathcal{X}$). A parent node possesses $K$ child nodes $\{(h+1,i_k)\}_{1\leq k \leq K}$, whose cells $\{\mathcal{X}_{h+1,i_k}\}_{1\leq k \leq K}$ form a partition of the parent's cell $\mathcal{X}_{h,i}$. The set of leaves in $\mathcal{T}$ is denoted as $\mathcal{L} \subseteq \mathcal{T}$. Attributes of a node $(h,i)$ are indexed by its $h$ and $i$. Accordingly, each node is associated with a representative state $\myvec{x}_{h,i} \in \mathcal{X}_{h,i}$ at which the objective function may be evaluated as a part of the sequential framework and out of the $v$-evaluation budget. Based on this evaluation, an optimistic bound of the function over $\mathcal{X}_{h,i}$, denoted by $b_{h,i}$ in analogy to the $B$-value in multi-armed bandits, is defined. The optimistic bound $b_{h,i}$ governs when $(h,i)$  gets expanded. Clearly, only evaluated leaf nodes are expandable and we denote them by $\mathcal{E}\subseteq \mathcal{L}$. The process of evaluating the function at $\myvec{x}_{h,i}$ is referred to as \textit{evaluating the node $(h,i)$}; and the process of splitting a cell $\mathcal{X}_{h,i}$, whose node $(h,i) \in \mathcal{E}$, into $K$ subcells (resp., $K$ child nodes) as \textit{expanding the node~$(h,i)$}.

\begin{figure}[tb]
	\begin{center}
		\renewcommand{\arraystretch}{1.2}
		\includegraphics[scale=1.]{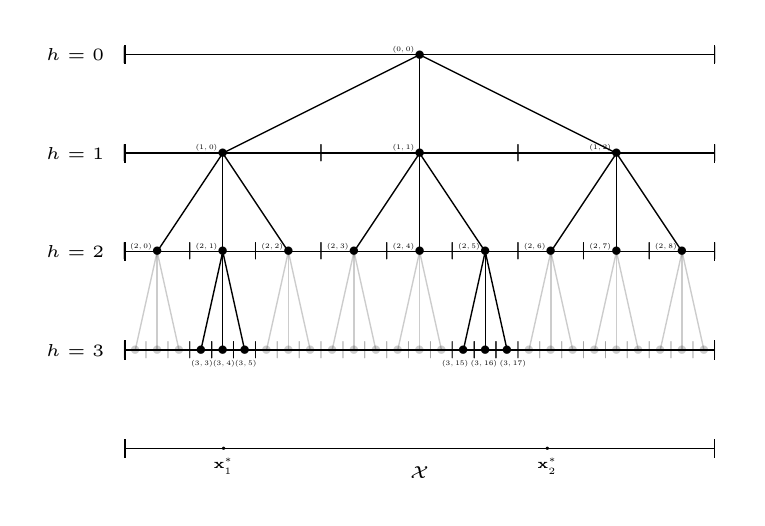}
	\end{center}
	\caption{Hierarchical partitioning of the decision space $\mathcal{X}$ with a partition factor of $K=3$ at iteration~$t$ represented by a $K$-ary tree. Consider a multi-objective problem where $m=2$ and the global optimizers of the first and second objective functions ($\myvec{x}^*_1$ and $\myvec{x}^*_2$, respectively) are shown. Thus, the nodes $\{(3,4),(2,1),(1,0),(0,0)\}$ and $\{(2,6),(1,2),(0,0)\}$ are $1$- and $2$-optimal nodes, respectively. Furthermore, $h^*_{1,t}=3$ and $h^*_{2,t}=2$ (more in Section~\ref{sec:conv}).}
	\label{fig:hierarchical_partition}		
\end{figure}

Among the several single-objective optimistic optimization algorithms that have been proposed and validated in the literature \cite{munos2011optimistic,valkostochastic,preux2014bandits,bbcomp15}; the Simultaneous Optimistic Optimization (\alg{SOO}) is the simplest, which makes it easy to implement efficiently. Furthermore, it is a rank-preserving algorithm, with theoretically provable finite-time performance~\cite{munos2011optimistic}, hence we are inspired by \alg{SOO} to solve the multi-objective optimization problem~\eqref{eq:problem_def}, optimistically.

\subsection{Simultaneous Optimistic Optimization (\alg{SOO})}
\label{sec:bg:soo}

The optimistic method, \alg{SOO}, was originally introduced in \cite{munos2011optimistic} and
falls in the family of global single-objective optimizers. It assumes local smoothness around the function's global minimum(a), \ie, $f(x) - f^* \leq \ell(x,x^*)$ where~$\ell:\mathcal{X}\times \mathcal{X}\to \mathbb{R}^+$ is a semi-metric. With this assumption, an optimistic lower bound of the objective function values over the cells  of the search-space hierarchical partitioning can be defined, mathematically:
\begin{equation}
	b_{h,i} = f(\myvec{x}_{h,i}) - \sup_{\myvec{x}\in \mathcal{X}_{h,i}} \ell(\myvec{x},\myvec{x}_{h,i}), \forall (h,i)\in \mathcal{T}
	\label{eq:b-value}
\end{equation}
Consequently, \texttt{SOO} would expand simultaneously all the nodes $(h,i)$ of its tree $\mathcal{T}$ whose $b$-values (Eq.~\ref{eq:b-value}) would be the least with respect to a semi-metric $\ell$.
However, in practice, the knowledge of $\ell$ is not always present.  Instead, \alg{SOO} simulates the effect of Eq.~\eqref{eq:b-value} by iteratively expanding at most a leaf node per depth if such node has the least $f(\myvec{x}_{h,i})$ with respect to leaf nodes of the same or lower depths. In addition to that, the algorithm takes a function $h_{max}(t)$, as a parameter,  such that after $t$ node expansions only nodes at depth~$h\leq h_{max}(t)$ can be expanded.

As outlined in Algorithm~\ref{alg:soo}, \alg{SOO} grows a tree $\mathcal{T}$ over $\mathcal{X}$ by expanding at most one leaf node per depth in an iterative sweep across $\mathcal{T}$'s depths/levels. At depth $h\geq0$, a leaf $(h,i)$ is expanded if its function value $f(\myvec{x}_{h,i})$ is the(or one of the) lowest (with respect to minimization) among the leaves at depth $h$ as well as all the expanded nodes at depths $<h$ in the current sweep. Splitting a node is worked out by partitioning its subspace along one dimension of $\mathcal{X}$, which can be chosen among $\mathcal{X}$'s dimensions either in a random (any one dimension out of the $n$) or sequential (one dimension after the other in a fixed sequence) manner. In \alg{SOO}, $\mathcal{E}=\mathcal{L}$, \ie, leaf nodes are evaluated once they are created.

\begin{algorithm}[tbp]
	\SetKwInOut{Input}{Input}
	\SetKwInOut{Variables}{Variables}
	\SetKwInOut{Output}{Output}
	\SetKwInOut{Initialization}{Initialization}
	
	\Input{function to be minimized $f$, \\
		search space $\mathcal{X}$,\\
		partition factor $K$, \\
		evaluation budget $v$, \\
		maximal depth function $t\to h_{max}(t)$
	}
	\Initialization{$\mathcal{T}_1\gets\{(0,0)\}$\\ $t\gets 1$}
	\Output{approximation of $\min_{\myvec{x}\in \mathcal{X}}f(\myvec{x})$}
	\BlankLine
	
	\While{evaluation budget is not exhausted} {
		$v \gets \infty$\\
		
		\For{$h\leftarrow 0$ \KwTo min$(h_{max}(t)$, depth($\mathcal{T}_t$))} {
			Among all leaves $(h, j)\in\mathcal{L}_t$ of depth $h$, select $(h, i) \in \arg\min_{(h,j)\in\mathcal{L}_t}
			f(\myvec{x}_{h,j})$ \label{line:soo_pt} \\
			\If{$f(\myvec{x}_{h,i}) \leq v$ \label{line:soo_qt}}
			{	
				$t\gets t+1 $\\
				$v \gets f(\myvec{x}_{h,i})$\\
				Expand the node $(h,i)$: add to $\mathcal{T}_t$ and evaluate its $K$ children,
				$\mathcal{T}_{t} \gets \mathcal{T}_{t-1} \cup \{(h+1,i_k)\}_{1\leq k\leq K}$\\
			}
		}
	}
	\Return{$\min_{\myvec{x}_{h,i}:(h,i)\in \mathcal{T}_t} f(\myvec{x}_{h,i})$}
	\caption{\textbf{\alg{SOO}}}
	\label{alg:soo}
\end{algorithm}

\section{Optimistic Optimization for Multi-Objective Problems}
\label{sec:mosoo}

This section presents the Multi-Objective Simultaneous Optimistic Optimization (\alg{MO-SOO}) algorithm to solve multi-objective optimization problems. \alg{MO-SOO} partitions the search space over multiple scales to find a good approximation set of the Pareto front. First, we describe a template for optimistic methods to address multiple objectives instead of a single objective and then present the \alg{MO-SOO} algorithm along with a worked example.

\subsection{From Single- to Multi-Objective Optimization}
\label{sec:from-soo-to-moo}
The class of optimistic methods encodes the search for optimal solutions as a tree of bandits, where the B-value of each arm represents an optimistic bound on the values of the objective function values over tree's nodes. In an iterative manner: an optimistic method assesses a set of leaf nodes of its tree on the search space~$\mathcal{X}$ and selectively expands a set of them. In other words, optimistic algorithms differ only in their strategies of growing and using the tree further to provide a good approximation of the optimal solutions. Based on this view, a generic template of optimistic algorithms for optimization problems can be derived (shown in Algorithm~\ref{alg:template_oo}), where the set of leaf nodes to be assessed at iteration $t$ are denoted by $\mathcal{P}_t$. Likewise, the set of leaf nodes to be expanded at iteration $t$ are denoted by $\mathcal{Q}_t\subseteq \mathcal{P}_t$.	In essence, $\mathcal{P}_t$ represents the subset of nodes that can be expanded at iteration $t$, which may depend on its depth/level. On the other hand, $\mathcal{Q}_t$ are the potentially optimal nodes according to their representative states that are expanded at iteration $t$. These two sets are algorithm-dependent.
\begin{algorithm}[tbp]
	\SetKwInOut{Input}{Input}
	\SetKwInOut{Variables}{Variables}
	\SetKwInOut{Output}{Output}
	\SetKwInOut{Initialization}{Initialization}
	\Initialization{$\mathcal{T}_t=\{(0,0)\}$\\ $t\gets 0$}
	\Output{approximation of optimal solutions}
	\BlankLine
	\While{evaluation budget is not exhausted} {
		Assess the nodes $\in \mathcal{P}_t$\\
		Expand the nodes $\in \mathcal{Q}_t$ \& add their child nodes in $\mathcal{T}_t$\\
		$t\gets t +1$\\
	}
	\Return{ the best solutions found}
	\caption{Template for Optimistic Optimization}
	\label{alg:template_oo}
\end{algorithm}

With regard to \alg{SOO}, $\mathcal{P}_t$ is the set of leaf nodes at the depth considered at iteration $t$ (Algorithm~\ref{alg:soo}, line~\ref{line:soo_pt}), whereas $\mathcal{Q}_t$ is at most one node $\in \mathcal{P}_t$ that satisfies the conditions in Algorithm~\ref{alg:soo}, lines \ref{line:soo_pt}--\ref{line:soo_qt}. On this notion of sets, optimistic methods can be extended to multi-objective settings by defining the corresponding $\mathcal{P}$ and $\mathcal{Q}$.

In other words and with regards to problem \eqref{eq:problem_def}, at the $s${th} step, choosing a node (resp., its representative state~$\myvec{x}^s$) depends on the previous~$s-1$ chosen nodes (resp., their representative states and corresponding objective vectors~$\{(\myvec{x}^1,\myvec{f}(\myvec{x}^1)),\ldots,(\myvec{x}^{s-1}\\,\myvec{f}(\myvec{x}^{s-1}))\}$). Consequently, the algorithm constructs a sequence of $v$ points and returns its approximation set, denoted by $\mathcal{Y}^v_*\in \Omega$.

In accordance with the single-objective loss measure for optimistic methods,\footnote{The quality of the returned solution for single-objective settings is evaluated by the loss measure:
	$r(v)= \min_{\myvec{x}\in\{\myvec{x}^1,\ldots,\myvec{x}^{v}\}}f(\myvec{x}) -  \min_{\myvec{x}\in\mathcal{X}} f(\myvec{x})$}
we introduce the following vectorial loss measure for \textsc{MOO}:
\begin{equation}
\myvec{r}(v)= \myvec{y}^v_* -  \myvec{y}^*
\label{eq:loss}
\end{equation}
where $\myvec{y}^v_*\myeq(\min_{\myvec{y}\in\mathcal{Y}^v_*} y_1,\ldots,\min_{\myvec{y}\in\mathcal{Y}^v_*} y_m)$ is the empirical ideal point found so far.

\subsection{The \alg{MO-SOO} Algorithm}
Based on the generic template of optimistic optimization (Algorithm~\ref{alg:template_oo}), an \textsc{MOO} algorithmic instance
whose aim is to recognize Pareto optimal solutions can be realized. Taking inspiration from \alg{SOO}, we refer to it as the Multi-Objective Simultaneous Optimistic Optimization (\alg{MO-SOO}). \alg{MO-SOO} iteratively considers leaf nodes, one depth at a time, starting from the root. The sets  $\mathcal{P}$ and $\mathcal{Q}$ are defined as follows. Denote $\mathcal{T}$'s depth considered at iteration $t$ by $h_t$, we have:
\begin{itemize}
	\item $\mathcal{P}_t \myeq \{\text{leaf nodes at depth } h_t\}$.
	\item $\mathcal{Q}_t \myeq$ the subset of nodes $\in \mathcal{P}_t$ that are non-dominated with respect to $\mathcal{P}_t$ as well as all the expanded nodes in the previous $h_t$ iterations, based on their representative objective vectors. Finding this set is captured by the operator $\ND(\cdot)$, which is defined next.
\end{itemize}

\begin{definition}[The non-dominated operator $\ND(\cdot)$] Let $A\subseteq \mathcal{Y}$ be a set of objective vectors. The operator $\ND(\cdot)$ is defined such that $\ND (A)$ is the set of all non-dominated vectors in $A$, \ie, 
	\begin{equation}
	\ND(A) \myeq \argmax\limits_{B\in \Omega, B\subseteq A} |B|\;,
	\end{equation}
	where $\Omega$ is the set of all possible approximation sets as stated by Definition~\ref{def:approx_set}.
	\label{def:nd_operator}
\end{definition}
The pseudo-code of the proposed scheme is outlined in Algorithm~\ref{alg:mosoo}. \alg{MO-SOO} comes with three parameters, viz. i) the partition factor $K$, ii) the maximal depth function $h_{max}(t)$, iii) the splitting dimension per depth. All of these parameters contribute to the algorithm exploration-vs.-exploitation trade-off. Nevertheless, as it will be shown later,  $h_{max}(t)$ has the most compelling impact on \alg{MO-SOO} convergence.

\begin{algorithm}[tbp]
	\SetKwInOut{Input}{Input}
	\SetKwInOut{Variables}{Variables}
	\SetKwInOut{Output}{Output}
	\SetKwInOut{Initialization}{Initialization}
	\Input{vectorial function to be minimized $\myvec{f}$, \\
		search space $\mathcal{X}$,\\
		partition factor $K$, \\
		evaluation budget $v$, \\
		maximal depth function $t\to h_{max}(t)$
	}
	\Initialization{$\mathcal{T}_1=\{(0,0)\}$\\ $t\gets 1$}
	\Output{approximation set of $\min_{\myvec{x}\in \mathcal{X}}\myvec{f}(\myvec{x})$, $\mathcal{Y}^v_*$}
	\BlankLine
	
	\While{evaluation budget is not exhausted} {
		$\mathcal{V} \gets \emptyset$\\
		\For{$h\leftarrow 0$ \KwTo min$(h_{max}(t)$, depth($\mathcal{T}_t$))} {
			$\mathcal{P}_t\gets \{\text{leaf nodes at depth }h\}$ \label{ln:beg_iteration}\\
			$\mathcal{V} \gets \ND(\mathcal{P}_t \cup \mathcal{V})$ \\
			$\mathcal{Q}_t \gets \mathcal{P}_t \cap \mathcal{V}$\\
			\label{algln:q_expand}Expand all the nodes in $\mathcal{Q}_t$; evaluate and add to $\mathcal{T}_t$ their $K\cdot|\mathcal{Q}_t|$ children,
			\begin{equation}
			\mathcal{T}_{t+1} = \mathcal{T}_{t} \cup \big(\cup_{(h,i)\in \mathcal{Q}_t}\{(h+1,i_k)\}_{1\leq k\leq K}\big)
			\nonumber
			\end{equation} \\

			$t\gets t+1 $\label{ln:end_iteration}\\
		}
	}
	\Return{$\ND\big({\{\myvec{f}(\myvec{x}_{h,i}) \}_{(h,i)\in \mathcal{T}_t}\big)}$}
	\caption{\textbf{\alg{MO-SOO}}}
	\label{alg:mosoo}
\end{algorithm}
\subsection{A Worked Example}
\label{sec:mosoo-example}

For a better understanding of the \alg{MO-SOO} algorithm, we show its application to the following problem:

\begin{equation}
\begin{aligned}
& {\text{minimize}}
& & \myvec{y} = \myvec{f}(\myvec{x}) = (f_1(\myvec{x}), f_2(\myvec{x}))\\
& \text{s.t.}
& & \myvec{x} =  (x_1, x_2)\in \mathcal{X}=[-1,1]^2\;, \\
\end{aligned}
\label{eq:worked_example}
\end{equation}
where $f_1(\myvec{x})=(x_1-0.25)^2+(x_2-0.66)^2$ and $f_2(\myvec{x})=(x_1+0.25)^2+(x_2-0.66)^2$. Figure~\ref{fig:worked_example} shows the convergence of \alg{MO-SOO}'s approximation set~$A$ towards a sampled set (numerically-obtained) of the Pareto front at different stages of the algorithm iterations. The reader can refer to Figure~\ref{fig:worked_example} as we briefly describe the first stages of the algorithm.

\begin{figure}[tbp]
	\begin{center}
		\resizebox{1\textwidth}{!}{
			\renewcommand{\arraystretch}{1.1}
			\begin{tabular}{|@{}c@{}@{}c@{}|@{}c@{}@{}c@{}|
				}
				\toprule
				\includegraphics[width=7cm, trim = 40mm 80mm 40mm 60mm, clip]{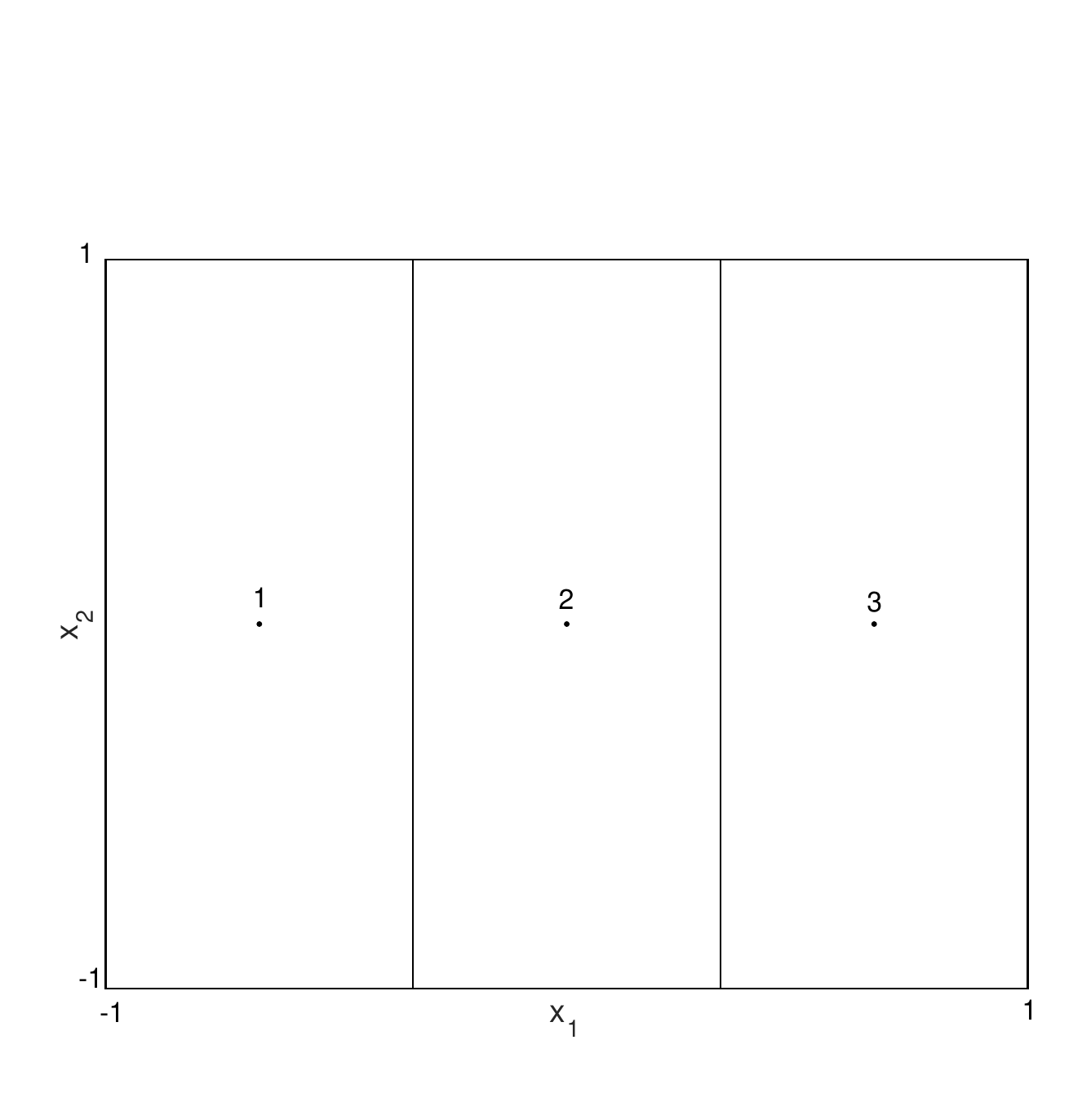} &
				\includegraphics[width=7cm, trim = 8mm 2mm 0mm 5mm, clip]{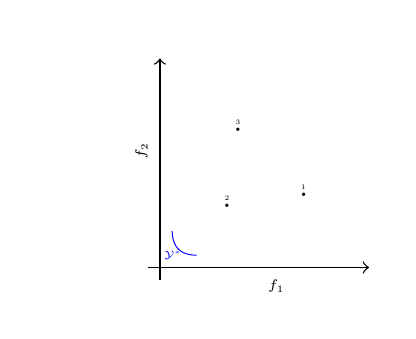}  &
				\includegraphics[width=7cm, trim = 40mm 80mm 40mm 60mm, clip]{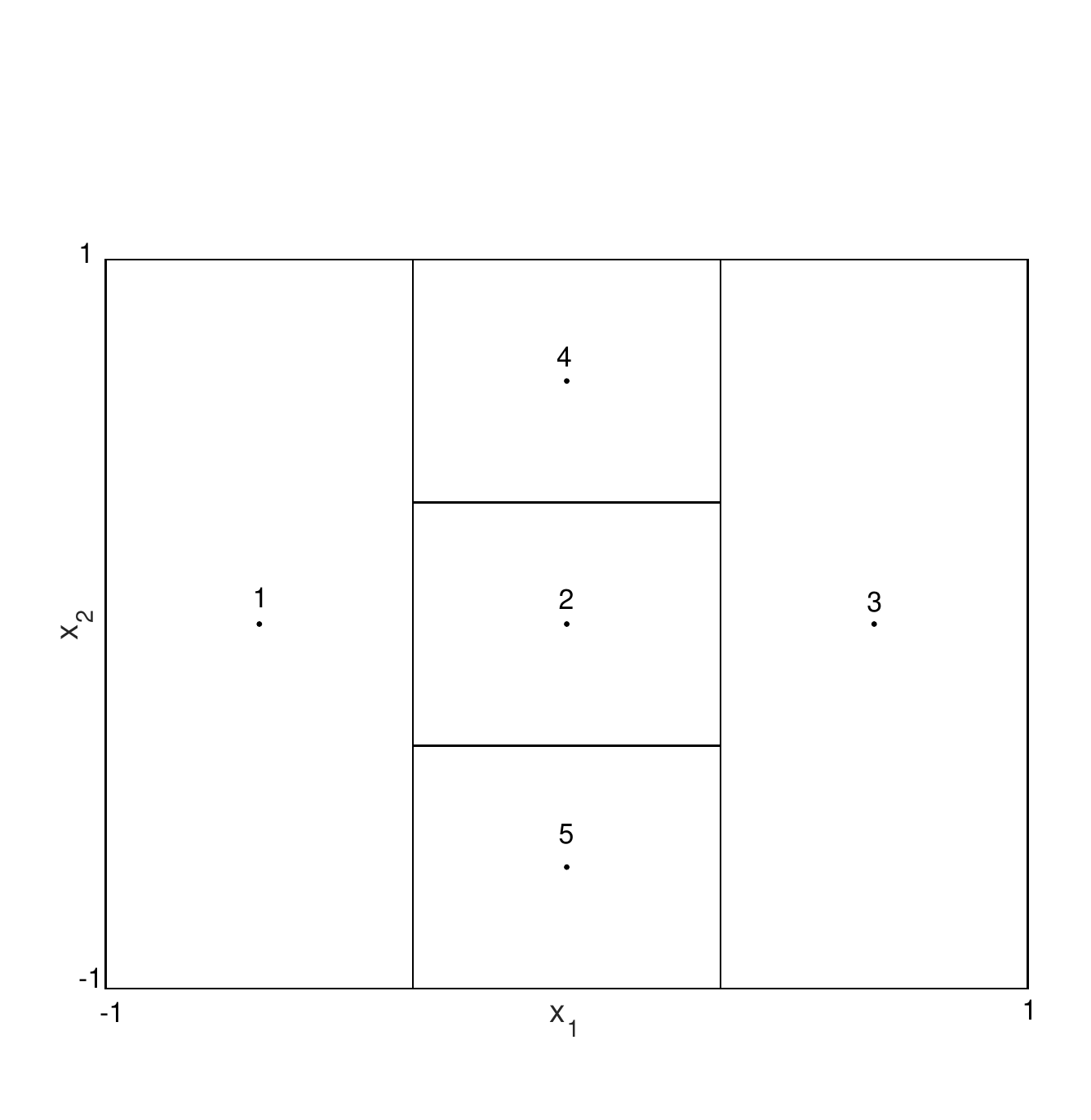} &
				\includegraphics[width=7cm, trim = 8mm 2mm 0mm 5mm, clip]{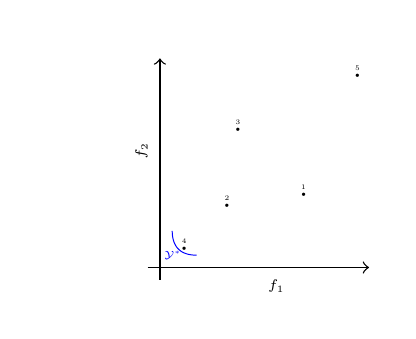} \\
				{Decision Space} & {Function Space} &{Decision Space} & {Function Space}\\										
				\multicolumn{2}{|c|}{\textbf{After 1 iteration}} & \multicolumn{2}{c|}{\textbf{After 2 iterations}}								
				\\	\midrule
				\includegraphics[width=7cm, trim = 40mm 80mm 40mm 60mm, clip]{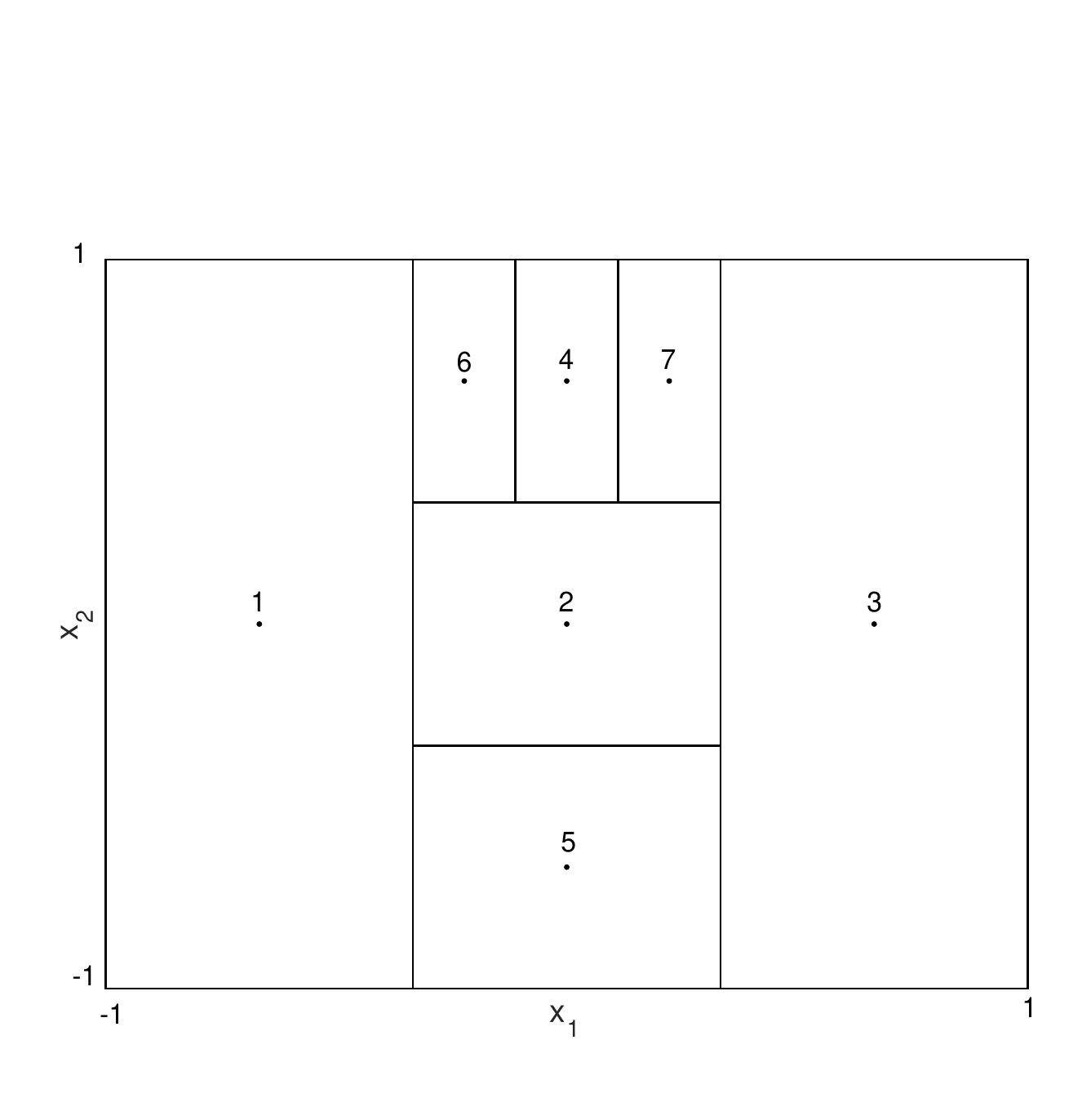} &
				\includegraphics[width=7cm, trim = 8mm 2mm 0mm 5mm, clip]{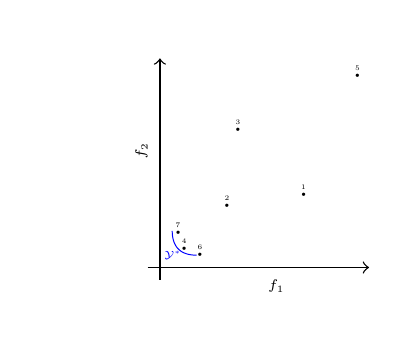}  &
				\includegraphics[width=7cm, trim = 40mm 80mm 40mm 60mm, clip]{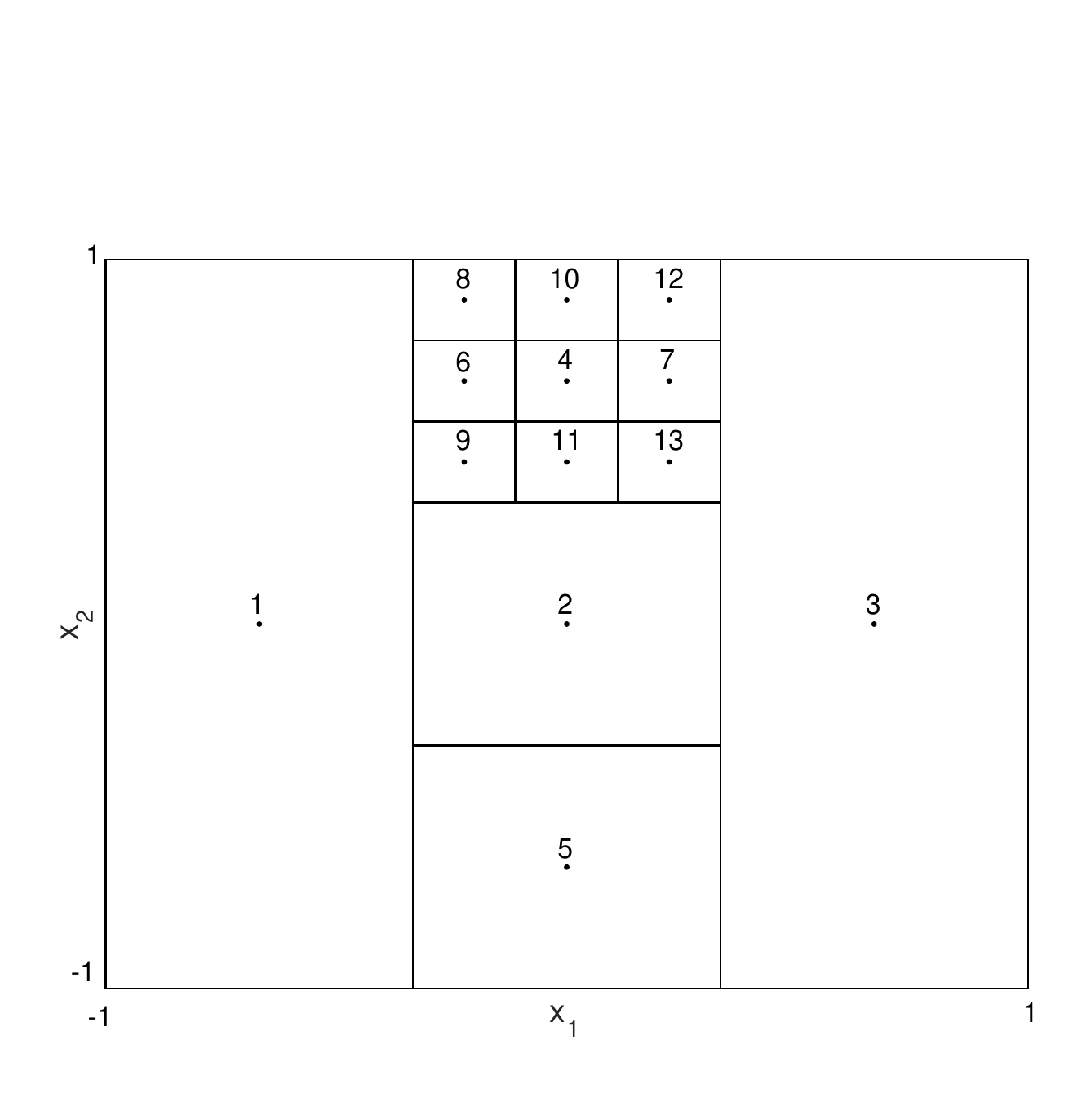} &
				\includegraphics[width=7cm, trim = 8mm 2mm 0mm 5mm, clip]{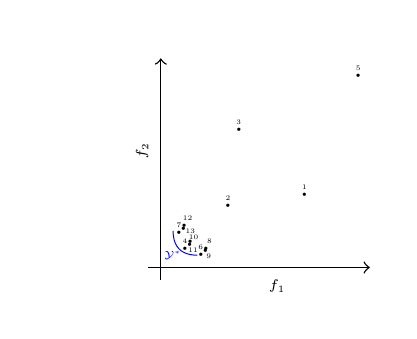}
				\\
				{Decision Space} & {Function Space} &{Decision Space} & {Function Space}\\										
				\multicolumn{2}{|c|}{\textbf{After 3 iterations}} & \multicolumn{2}{c|}{\textbf{After 4 iterations}}								
				\\	\midrule
				\includegraphics[width=7cm, trim = 40mm 80mm 40mm 60mm, clip]{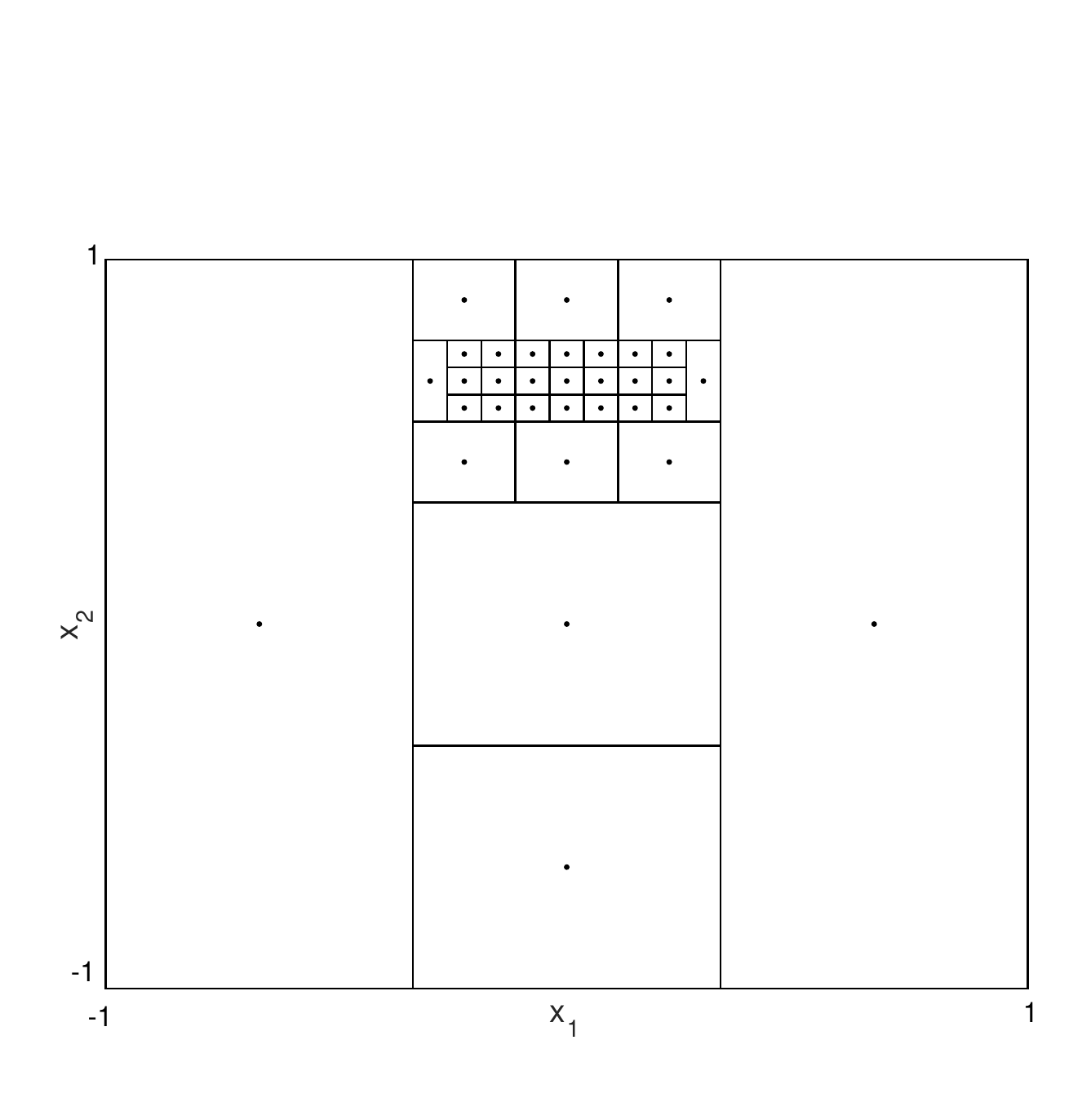} &
				\includegraphics[width=7cm, trim = 8mm 2mm 0mm 5mm, clip]{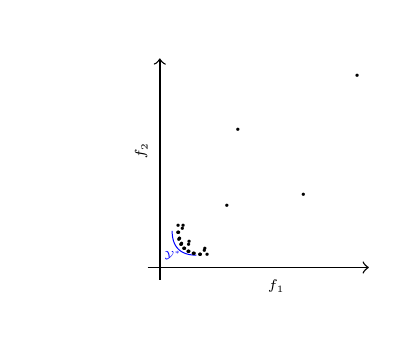}  &
				\includegraphics[width=7cm, trim = 40mm 80mm 40mm 60mm, clip]{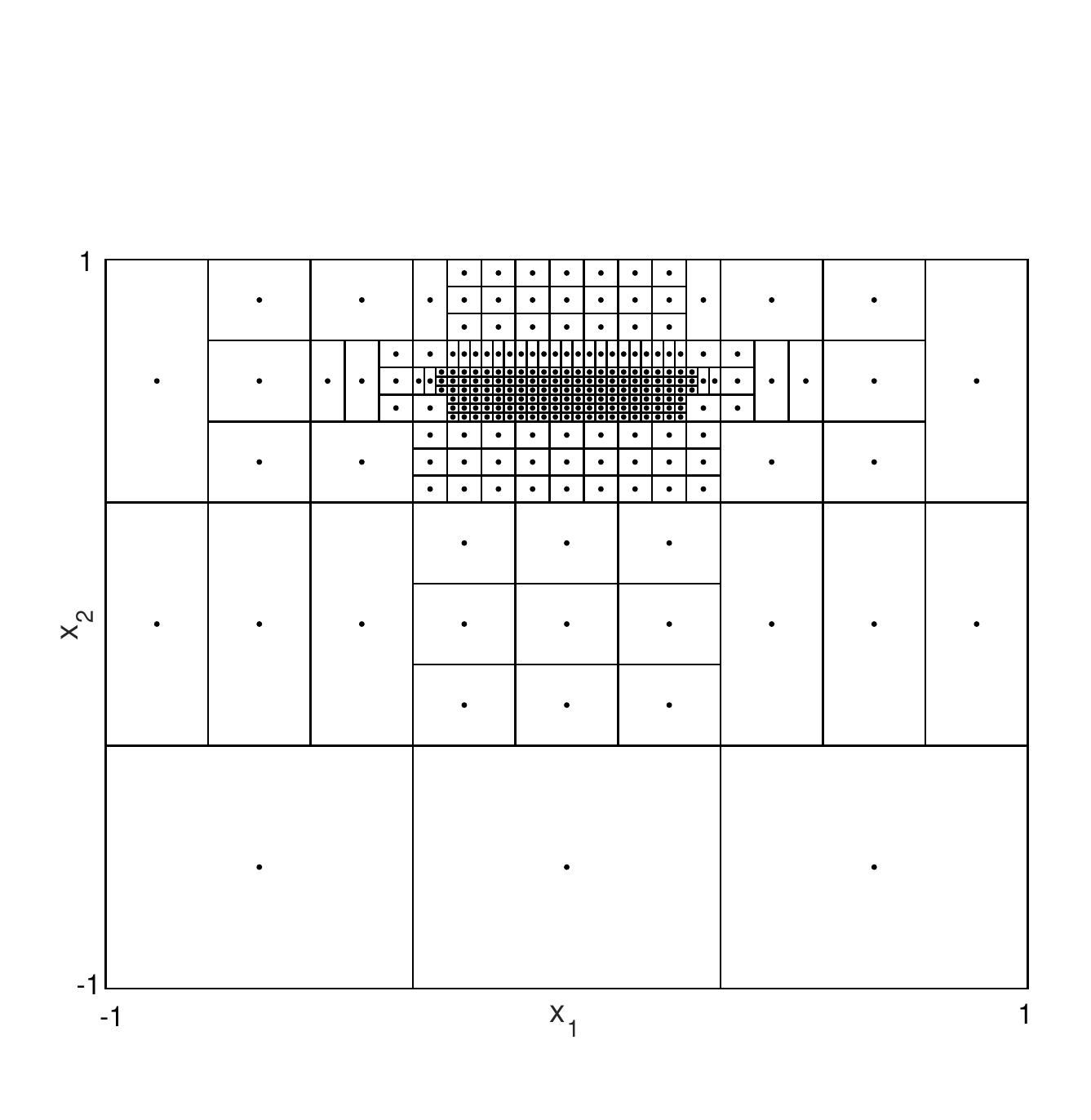} &
				\includegraphics[width=7cm, trim = 8mm 2mm 0mm 5mm, clip]{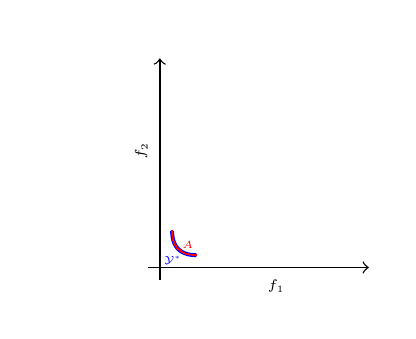}
				\\
				{Decision Space} & {Function Space} &{Decision Space} & {Function Space}\\										
				\multicolumn{2}{|c|}{\textbf{After 6 iterations}} & \multicolumn{2}{c|}{\textbf{After 20 iterations}}	\\									
				\bottomrule														
		\end{tabular}}
	\end{center}
	\caption{An illustration of \textsc{MO-SOO}. The algorithm expands its leaf nodes to look for Pareto optimal solutions by partitioning their cells along the decision space coordinates one at a time in a sequential manner, with a partition factor of~$K=3$. Sweeping its tree from the root node depth till the maximal depth specified by $h_{max}(t)$. \alg{MO-SOO} expands a set of leaf nodes per depth if they are non-dominated with respect to other leaf nodes in the same level and with respect to those expanded at lower depths in the current sweep. 
		Subsequently, none or more nodes can be expanded in one iteration; at the third iteration, for instance, only one node is expanded (whose representative state is point~4) into its children (whose representative states are the points 4, 6, and 7). On the other hand, at the fourth iteration three nodes are expanded (resp., representative states are the points 4, 6, and 7) into their children nodes (resp., representative states are the points 4, 6, 7, 8, 9, 10, 11, 12, and 13). After 20 iterations, \alg{MO-SOO}'s approximation set~$A$ closely coincides on a sampled set of the Pareto front~$\mathcal{Y}^*$ of problem~\eqref{eq:worked_example}. }
	
	\label{fig:worked_example}							
\end{figure}

\textit{Initialization}. \alg{MO-SOO} starts by initializing its tree with a root node $(0,0)$ whose cell represents the decision space, \ie, $\mathcal{X}_{0,0}=\mathcal{X}$. The root's representative state ~$\myvec{x}_{0,0}=(0,0)$---point 2 in Figure~\ref{fig:worked_example}---is evaluated and $\myvec{f}(\myvec{x}_{0,0})$ is obtained.

\textit{Iteration 1}. At this iteration, leaf nodes at depth~$h=0$ are considered for expansion. In other words,  the root node is expanded by partitioning its cell along the first dimension of the decision space into $K=3$ cells. Here, $\mathcal{P}_1=\mathcal{Q}_1=\{(0,0)\}$. For convenience, we shall refer to the nodes by their representative states, \ie, $\mathcal{P}_1=\mathcal{Q}_1=\{2\}$. The newly generated leaf nodes are added to the tree and evaluated at their representative states viz. the points 1, 2, and 3 in Figure~\ref{fig:worked_example}. Note that having an odd partition factor ($K=3$) saves one function evaluation for each node expansion (point 2 was already evaluated).

\textit{Iteration 2}. At this iteration, leaf nodes at depth~$h=1$  are considered for expansion. We have $\mathcal{P}_2=\{1,2,3\}$ and $\mathcal{V}=\{2\}$. Along with Lines~\ref{ln:beg_iteration}--\ref{ln:end_iteration} of Algorithm~\ref{alg:mosoo}, $\mathcal{Q}_2$ becomes $\{2\}$, because point 2 dominates both points 1 and 3 as it can noted in the function space. Thus, node 2 is expanded and the tree grows to have the leaves~$\mathcal{L}=\{1,\ldots,5\}$, each being evaluated at its representative state. The case is the same for \textit{iteration 3} which considers nodes at $h=2$ generating a new set of leaves $\mathcal{L}=\{1,\ldots,7\}$.

\textit{Iteration 4}.  At this iteration, leaf nodes at depth~$h=3$ are considered  for expansion.  Here, $\mathcal{P}_4=\{4,6,7\}$ and $\mathcal{V}=\{4,6,7\}$. Along with Lines~\ref{ln:beg_iteration}--\ref{ln:end_iteration} of Algorithm~\ref{alg:mosoo}, $\mathcal{Q}_4$ becomes $\{4,6,7\}$, because the points 4, 6, and 7 are non-dominated with respect to the nodes in $\mathcal{P}_4$ and $\mathcal{V}$ as it can seen in the function space. Thus, they all are expanded and the tree grows to have the leaves~$\mathcal{L}=\{1,\ldots,13\}$.

\textit{Next Iterations}. The same holds for the next iterations until the maximal depth---specified by $h_{max}(t)$---is reached. Then, $\mathcal{V}$ is set to $\emptyset$ and the tree is swept again from its root. After some iterations, \alg{MO-SOO} closely approximates the Pareto front as shown in Figure~\ref{fig:worked_example}.

\section{Convergence Analysis}
\label{sec:conv}
The analysis of multi-objective solvers is hindered by several issues; namely the diversity of approximation sets, the complexity of the Pareto front, and the convergence of approximation sets to the Pareto front \cite{coello2002evolutionary}. While most theoretical convergence studies have addressed finite-set and/or discrete problems \cite{rudolph1998evolutionary,kumar2005running}, others have provided probabilistic guarantees~\cite{hanne1999convergence}, assumed a total order on the solutions \cite{gabor1998multi}, or studied their asymptotic behavior~\cite{custodio2011direct}. In this paper, we take a different approach and study \alg{MO-SOO}'s convergence in terms of two aspects: i). finite-time; and ii). asymptotic behavior.

First, the finite-time convergence of \alg{MO-SOO} is studied with respect to the Pareto-compliant quality indicator,\footnote{The quality of an approximation set is measured by a so-called (unary) quality indicator $I:\Omega\to \mathbb{R}$, assessing a specific property of the approximation set. Likewise, an $l$-ary quality indicator $I:\Omega^l\to \mathbb{R}$ quantifies quality differences between $l$ approximation sets~\cite{zitzler2003performance,custodio2011direct}. A quality indicator is not Pareto-compliant if it contradicts the order induced by the Pareto-dominance relations.} the unary additive epsilon indicator \epsind~\cite{zitzler2003performance}, based on three assumptions. We do this in a two-step approach. First, we upper bound the loss measure introduced in Section~\ref{sec:from-soo-to-moo}, viz. $\myvec{r}(v)$ of  Eq.~\eqref{eq:loss}. The loss measure captures the convergence of \alg{MO-SOO}'s approximation set~$\mathcal{Y}^v_*$ to $m$ points---on the Pareto front---that contribute to the problem's ideal point $\myvec{y}^*$. Second, based on the presented loss bound and an intrinsic measure of the Pareto front (we refer to this measure as the conflict dimension $\Psi$), an upper bound on the unary additive epsilon indicator~\epsind~is established. Second, the convergence of \alg{MO-SOO}'s approximation set towards the Pareto front given unlimited number of function evaluations is addressed. In the light of the assumptions made for the finite-time analysis, \alg{MO-SOO}'s consistency is investigated. An algorithm is said to be consistent if it asymptotically converges to the Pareto front.

In general, the design of optimistic algorithms is driven by assumptions about the function smoothness. Here, we make three assumptions about the function $\myvec{f}$ and the hierarchical partitioning, based on those presented in \cite{munos2011optimistic,valkostochastic,wang2014bayesian} for single-objective settings. In essence, these assumptions let us express the quality of \alg{MO-SOO} solutions in relation to the number of iterations, by quantifying how much exploration is needed to expand nodes that contain objective-wise optimal solutions. The rest of this section is organized as follows. First, these assumptions are stated in Section~\ref{sec:assmptns}. Then,  in Section~\ref{sec:finite_time_performance}, the finite-time performance of \alg{MO-SOO} is analyzed, where we first upper bound the loss \eqref{eq:loss} as a function of the number of iterations~$t$.\footnote{Typically, $v$ in Eq.~\eqref{eq:loss} and the approximation set~$\mathcal{Y}^v_*$ represents the number of sampled points (function evaluations). Nevertheless, one can express the loss (and likewise the approximation set) with other growing-with-time quantities (\eg, the number of iterations, the number of node expansions). In the rest of this paper, we refer to the number of the: function evaluations and iterations, by $v$ and $t$, respectively, where one iteration represents executing the lines~\ref{ln:beg_iteration}--\ref{ln:end_iteration} of Algorithm~\ref{alg:mosoo}, once.} Second, this objective-wise loss bound is employed to establish an upper bound on the~\epsind~indicator, which holds down to the conflict dimension of the problem at hand. After presenting the main result on the finite-time performance of the algorithm, \alg{MO-SOO}'s consistency property is proved in Section~\ref{sec:asymptotic} and illustrative examples are given in Section~\ref{sec:illust}. Towards the end of this section, an empirical validation of the theoretical findings is presented.

\subsection{Assumptions} 
\label{sec:assmptns}
There exists a vector-valued function $\Bell:\mathcal{X}\times \mathcal{X}\to {\mathbb{R}^+}^m$ such that each entry $\{\ell_j\}_{1\leq j \leq m}$ is a semi-metric 
such that:
\begin{enumerate}[label=A\arabic*]
	\item \label{asmp:smoothness}\textbf{(H\"{o}lder continuity of $f_1,\ldots,f_m$)}:
	
	\begin{equation}
	|f_j(\myvec{x}) - f_j(\myvec{y})| \leq \ell_j(\myvec{x},\myvec{y}),\;\;\forall \myvec{x},\myvec{y}\in\mathcal{X}, j=1,\ldots,m\;\;.\nonumber
	\end{equation}
	\item \label{asmp:bounded_cell}\textbf{(bounded cells diameters)}: For $j=1,\ldots,m$ and $\forall (h,i) \in \mathcal{T}$,
	$\exists$ a non-increasing sequence $\delta_j(h)>0$ such that
	\begin{equation}
	\sup_{\myvec{x}\in\mathcal{X}_{h,i}} \ell_j(\myvec{x}_{h,i},\myvec{x})\leq \delta_j(h)
	\nonumber
	\end{equation}  and $\lim_{h\to \infty}\delta_j(h)=0$. Thus, ensuring the regularity of the cells' sizes which decrease with their depths in $\mathcal{T}$.
	\item \label{asmp:shaped_cells}\textbf{(well-shaped cells)}: For $j=1,\ldots,m$ and $\forall (h,i) \in \mathcal{T}$,
	$\exists$ $s_j>0$ such that a cell $\mathcal{X}_{h,i}$ contains an $\ell_j$-ball of radius $s_j\delta_j(h)$ centered in $\myvec{x}_{h,i}$. Thus, ensuring that the cells' shapes are not skewed in some dimensions.
\end{enumerate}

\begin{remark}
	The class of functions that satisfies Assumption~\ref{asmp:smoothness} is very broad. In fact, it has been shown in \cite{pinter1995,modifications_direct} that among the Lipschitz-continuous functions (which are a subset of such functions) are convex/concave functions over a closed domain and continuously differentiable functions.
	\label{rem:class_holder_fn}
\end{remark}
\subsection{Finite-Time Performance}
\label{sec:finite_time_performance}
In this section, we characterize the finite-time performance of \alg{MO-SOO} in terms of the Pareto-compliant unary additive  epsilon indicator based on the assumptions presented in Section~\ref{sec:assmptns}. To this end, we upper bound the loss measure~\eqref{eq:loss} with respect to the number of iterations~$t$. 	This provides the basis upon which a bound for the $\epsilon$-indicator is established with respect to the same.

\subsubsection{Bounding the Loss Measure}
\label{sec:bounding_loss}
In order to derive a bound on the loss, we employ a measure of the quantity of objective-wise near-optimal solutions (states in $\mathcal{X}$), called the \textit{near-optimality dimension}, which is closely related to similar measures (see, \eg,\cite{kleinberg2008multi,bubeck2009online,munos2011optimistic}). Before defining the near-optimality dimension, some terminology, which will be used in the analysis besides the terminology of Section~\ref{sec:bg:bbmo}, is introduced.

For $j=1,\ldots,m$; and for any $\epsilon>0$; let us denote the set of $\epsilon$-optimal \textit{states} according to $f_j$, $\{\myvec{x} \in \mathcal{X} : f_j(\myvec{x})\leq f_j(\myvec{x}^*_j)+\epsilon \}$, by $\mathcal{X}_j^\epsilon$, as depicted in Figure~\ref{fig:conv_preliminary}. Subsequently, denote the set of \textit{nodes} at depth $h$ whose representative states are in $\mathcal{X}_j^{\delta_j(h)}$ by $\mathcal{I}_j^h$, \ie, $\mathcal{I}_j^h\myeq\{(h,i)\in \mathcal{T}: 0\leq i\leq K^h-1, \myvec{x}_{h,i}\in\mathcal{X}_j^{\delta_j(h)} \}$. 
A node $(h,i)$ is Pareto optimal $\iff \exists \myvec{x} \in \mathcal{X}^* : \myvec{x}\in \mathcal{X}_{h,i}$. Furthermore, a Pareto optimal node $(h,i)$ is $j$-optimal $\iff$ it is optimal with respect to $f_j$, \ie, $\myvec{x}^*_j \in \mathcal{X}_{h,i}$. After $t$ iterations, one can denote the depth of the \text{deepest expanded $j$-optimal} node by $h^{*}_{j,t}$ (as illustrated in Figure~\ref{fig:hierarchical_partition}).
Now, we define the near-optimality dimension for~$f_j$: 
\begin{definition}[$s_j$-near-optimality dimension]
	The $s_j$-near-optimality dimension for $\{f_j\}_{1\leq j \leq m}$ is the smallest $d_{s_j} \geq 0$ such
	that there exists $C_j > 0$ and for any $\epsilon > 0$, the maximal number of disjoint $\ell_j$-balls of radius $s_j\epsilon$
	and center in $\mathcal{X}_j^\epsilon$ is less than $C_j{\epsilon}^{-d_{s_j}}$.
	\label{def:near_opt_dim}
\end{definition}

One can note that $d_{s_j}$ is characterized by: the function~$f_j$, the semi-metric~$\ell_j$, and the scaling factor $s_j$, \ie, it depends on the objectives smoothness and related to the partitioning strategy of the space through the scaling factor $s_j$
. Based on Assumption~\ref{asmp:shaped_cells} and Definition~\ref{def:near_opt_dim}, we have:

\begin{equation}
|\mathcal{I}^h_j|\leq C_j{\delta_j(h)}^{-d_{s_j}}\;.\footnote{ See the proof of \cite[Lemma 1]{munos2011optimistic}. We reproduce and adapt the proof here for completeness: From Assumption~\ref{asmp:shaped_cells}, each cell $\mathcal{X}_{h, i}$ contains a ball of radius $s_j\delta_j(h)$ centered in $x_{h,i}$, thus if
	$|\mathcal{I}^h_j| = |\{x_{h,i} \in \mathcal{X}^{\delta_j(h)}_j\}|$ exceeded $C_j{\delta_j(h)}^{-d_{s_j}}$, this would mean that there exists more than $C_j{\delta_j(h)}^{-d_{s_j}}$
	disjoint $\ell_j$-balls of radius $s_j\delta_j(h)$ with center in $\mathcal{X}^{\delta_j(h)}_{j}$, which contradicts the definition of $s_j$-near-optimality-dimension.}
\label{eq:optimal_nodes_bound}
\end{equation}
\begin{figure}[tb]
	\begin{center}
		\renewcommand{\arraystretch}{1.2}
		\includegraphics[scale=1.2,trim=5 10 2 10 mm, clip=true]{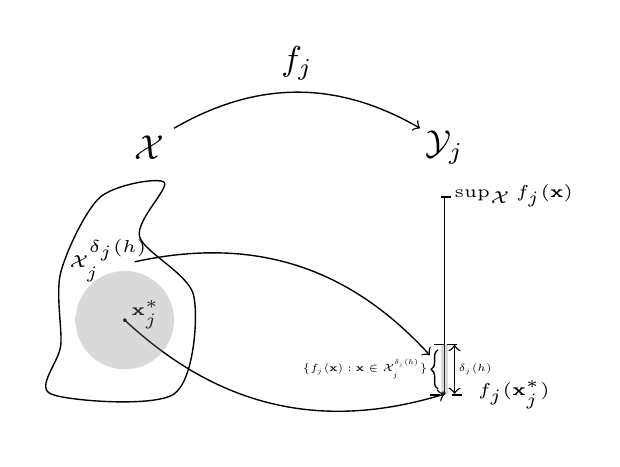}
	\end{center}
	\caption{The feasible decision space and the corresponding $j$th objective space ($\mathcal{Y}_j\subseteq \mathbb{R}$). The global optimizer $\myvec{x}^*_j$ and any solution $\myvec{x}$ whose image under the $j$th objective lies within $\{f_j(\myvec{x})\leq f_j(\myvec{x}^*_j) + \delta_j(h)\}$ are denoted by $\mathcal{X}_j^{\delta_j(h)}$.}
	\label{fig:conv_preliminary}		
\end{figure}

\begin{figure}[tb]
	\begin{center}
		\renewcommand{\arraystretch}{1.2}
		\includegraphics[scale=0.75,trim=5 10 2 10 mm, clip=true]{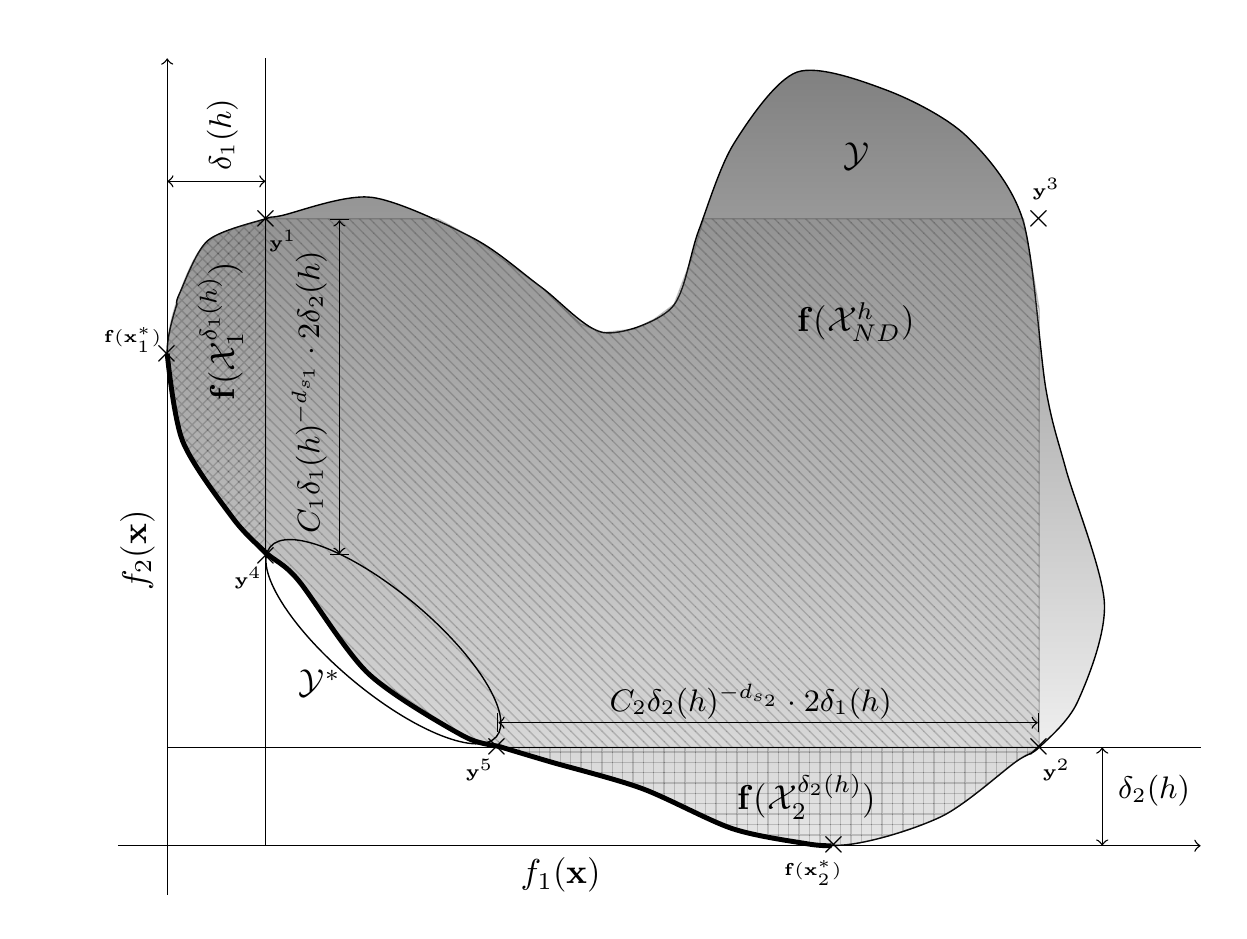}
	\end{center}
	\caption{The objective space~$\mathcal{Y}$ for a multi-objective optimization problem ($m=2$). The solid curve marks the Pareto front $\mathcal{Y}^*$. At depth $h$, assuming the $\{j\}_{j=1,2}$-optimal nodes are not expanded yet, one can use the $\mathtt{ND}_{min}(\cdot)$ operator, which causes nodes whose representative states lie in the decision space portion $\cup_{j={1,2}} \mathcal{X}^{\delta_j(h)}_j$ to be expanded before others. However, this may hold up discovering other parts of the Pareto front (the circled region). This is not the case with the $\texttt{ND}(\cdot)$ operator, where another region $\mathcal{X}^h_{ND}$---whose image in the objective space is bounded by the Pareto front and the points $\mathbf{y}^1$, $\mathbf{y}^2$, and $\mathbf{y}^3=\myvec{y}^{nadir}(\myvec{f}(\cup_{j=1,2}\mathcal{X}^{\delta_j(h)}_j))$---is as well considered. Let the considered depth at iteration $t$ be $h$ and the depth of the deepest $\{j\}_{j=1,2}$-optimal nodes be $h-1$. Then, prior to expanding the $\{j\}_{j=1,2}$-optimal nodes at depth $h$, the set $\mathcal{Q}_t$ (of Algorithm~\ref{alg:mosoo}) comprises of at most $3$ types of nodes whose representative states lie in $\mathcal{X}^{\delta_1(h)}_1, \mathcal{X}^{\delta_2(h)}_2$, and $\mathcal{X}^h_{ND}$, respectively. Furthermore, one can note from Eq.~\eqref{eq:optimal_nodes_bound} as well as Assumptions~\ref{asmp:smoothness} and~\ref{asmp:bounded_cell} that the point $\myvec{y}^1$  is greater than or equal $\myvec{y}^4$, and hence $\myvec{f}(\myvec{x}^*_1)$, along $f_2$ by at most $C_1\delta_1(h)^{-d_{s_1}}\cdot2\delta_2(h)$, that is to say $y^1_2-f_2(\myvec{x}^*_1)\leq C_1\delta_1(h)^{-d_{s_1}}\cdot2\delta_2(h)$; similar argument can be made between the points $\myvec{y}^2$ and $\myvec{y}^5$ along $f_1$. This observation is the main ingredient in the proof of Theorem~\ref{thm:indicator_mosoo} (more in Section~\ref{sec:bounding_additive_indicator}).
	}
	\label{fig:conv_preliminary_1}		
\end{figure}
Now, let us assume for simplicity that the $\ND(\cdot)$ operator in Algorithm~\ref{alg:mosoo} is replaced by $\ND_{min}(A)=\cup_{1\leq j\leq m}\arg\min_{\myvec{y}\in A}y_j$; that is to say, in each iteration, $m$ or less nodes are expanded whose representative objective vectors $\myvec{f}(\myvec{x}_{h,i})$ have the minimum entries with respect to the $m$ objectives. Furthermore, for $j=1,\ldots, m$; assume that $h^*_{j,t}=\acute{h}$ and denote the $j$-optimal node at depth $\acute{h}+1$ by $(\acute{h}+1,j^*)$. Since $(\acute{h}+1,j^*)$ has not been expanded yet, any node at depth $\acute{h}+1$ that is selected at later iterations  and expanded before $(\acute{h}+1, j^*)$ (line~\ref{algln:q_expand} in Algorithm~\ref{alg:mosoo}) must satisfy the following:
\begin{align}
f_j(\myvec{x}_{\acute{h}+1,i}) &\leq  f_j(\myvec{x}_{\acute{h}+1,j^*}) \nonumber \\
f_j(\myvec{x}_{\acute{h}+1,i})& \leq  f_j(\myvec{x}^*_j) + \delta_j(\acute{h}+1) \label{eq:optimal_nodes}
\end{align}
where inequality ~\eqref{eq:optimal_nodes} comes from combining Assumptions \ref{asmp:smoothness} and \ref{asmp:bounded_cell}: $f_j(\myvec{x}_{\acute{h}+1,j^*})\leq f_j(\myvec{x}^*_j) + \ell_j(\myvec{x}_{\acute{h}+1,j^*},\myvec{x}^*_j) \leq  f_j(\myvec{x}^*_j) + \delta_j(\acute{h}+1)$. As defined earlier, $\mathcal{X}^{\delta_j(h)}_j$ satisfies Eq.~\eqref{eq:optimal_nodes} (depicted in Figure~\ref{fig:conv_preliminary_1}, for $m=2$).
Thus, from the definition of $\mathcal{I}^h_j$ and since all the objectives are considered simultaneously, we are certain that $\{(\acute{h}+1,j^*)\}_{1\leq j \leq m}$ get expanded after  $\sum_{j=1}^{m}|\mathcal{I}^{\acute{h}+1}_j|$ node expansions at depth $\acute{h}+1$ in the worst-case scenario. Nevertheless, such definition of the $\ND_{min}(\cdot)$ operator favors exploring $\{\mathcal{X}^{\delta_j(h)}_j\}_{1\leq j \leq m}$ over other regions, which delays the search for other Pareto points outside these regions (see, for instance, the circled region in Figure~\ref{fig:conv_preliminary_1}).   Using the $\ND(\cdot)$ operator from Definition~\ref{def:nd_operator} rectifies this behavior: by expanding non-dominated nodes, \alg{MO-SOO} explores as well the region $\{\myvec{x}: \myvec{f}(\myvec{x}) \prec \myvec{y}^{nadir}(\myvec{f}(\cup_{j=1,2}\mathcal{X}^{\delta_j(h)}_j))\}-\cup_{j=1,2}\mathcal{X}^{\delta_j(h)}_j$ denoted by $\mathcal{X}^h_{ND}$ (see Figure~\ref{fig:conv_preliminary_1}). While we are able to quantify---based on the near-optimality dimension---the number of nodes within $\{\mathcal{X}_j^{\delta_j(h)}\}_{1\leq j \leq m}$, similar analysis gets unnecessarily complicated for $\mathcal{X}^h_{ND}$. However, since $\ND(\cdot)$ expands---besides other nodes---the same set of nodes that would have been selected by $\ND_{min}(\cdot)$, we know that at most $|\mathcal{I}^{\acute{h}+1}_j|$ iterations at depth $\acute{h}+1$ are needed to expand the optimal node $(\acute{h}+1,j^*)$.   From this observation, the following lemma is deduced.
\vspace{1mm}
\begin{lem} In \alg{MO-SOO}, after $t$ iterations, for any depth $0\leq h \leq h_{max}(t)$ whenever
	\begin{equation}
	h_{max}(t)\cdot \sum_{l=0}^{h}\max_{1\leq j\leq m}|\mathcal{I}^l_j|\leq t\;, \label{eq:lm_mosoo}
	\end{equation}
	we have $\{h^*_{j,t}\}_{1 \leq j \leq m}\geq h$.
	\label{lm:mosoo}
\end{lem}
\begin{proof}We know that $\{h^*_{j,t}\}_{1\leq j \leq m}\geq 0$ and hence the above statement holds for $h=0$. For $0<h\leq h_{max}(t)$, we are going to prove it by induction. 

Assume that the statement holds for $0\leq h\leq \hat{h}<h_{max}(t)$. Let us then prove it for $h\geq\hat{h}+1$. Let $h_{max}(t)\cdot \sum_{l=0}^{\hat{h}+1}\max_{1\leq j\leq m}|\mathcal{I}^l_j|\leq t$, and hence, $h_{max}(t) \cdot \sum_{l=0}^{\hat{h}}\max_{1\leq j\leq m}|\mathcal{I}^l_j|$ is less than or equal to $t$ for which we know by our assumption that $\{h^*_{j,t}\}_{1\leq j \leq m}\geq \hat{h}$. Here, we have two cases: (i) $\{h^*_{j,t}\}_{1\leq j \leq m}\geq \hat{h}+1$, for which the proof is done; (ii) $\{h^*_{j,t}\}_{1\leq j \leq m}=\hat{h}$, for this case, the set of nodes expanded at depth $\hat{h}+1$ at each iteration, before the $\{j\}_{1\leq j \leq m}$-optimal nodes at the same depth, belong to $m+1$ sets (possibly overlapped) of nodes. Among theses sets, $m$ sets are from $\{\mathcal{I}^{\hat{h}+1}_j\}_{1\leq j \leq m}$, respectively; while the remaining set of nodes have their representative states in $\mathcal{X}^{\hat{h}+1}_{ND}-\cup_{j=1}^{m}\mathcal{X}^{\delta_j(\hat{h}+1)}_j$. As a result, at each iteration, there could be at least one node to be expanded from $\{\mathcal{I}^{\hat{h}+1}_j\}_{1\leq j\leq m}$, respectively. Since expanding all of these nodes takes at most $\max_{1\leq j\leq m}|\mathcal{I}^{\hat{h}+1}_j|$ iterations at depth $\hat{h}+1$; with a tree of depth $h_{max}(t)$, we are certain that the $\{j\}_{1\leq j \leq m}$-optimal node at depth $\hat{h}+1$ are expanded after at most $h_{max}(t)\max_{1\leq j\leq m}|\mathcal{I}^{\hat{h}+1}_j|$ iterations. Therefore, we have~$\{h^*_{j,t}\}_{1 \leq j \leq m}\geq h$.\end{proof}


In other words, the size of $\mathcal{I}^h_j$ gives a measure of how much exploration is needed, provided that the $j$-optimal node at depth $h-1$ has expanded; and this exploration is quantified by the near-optimality dimension.
The next theorem builds on Lemma \ref{lm:mosoo} to present a finite-time analysis of \alg{MO-SOO} in terms of a bound on the loss of Eq.~\eqref{eq:loss} as a function of the number of iterations $t$.

\begin{thm}[$\myvec{r}(t)$ for \alg{MO-SOO}] Let us define $h(t)$ as the smallest $h\geq0$ such that:	
	\begin{equation}
	h_{max}(t)\sum\limits_{l=0}^{h(t)}\max_{1\leq j \leq m}	C_j\delta_j(l)^{-d_{s_j}} \geq t \label{eq:thm_mosoo}
	\end{equation}
	where $t$ is the number of iterations.
	Then the loss of \alg{MO-SOO} is bounded as:
	\begin{equation}
	r_j(t)\leq \delta_j(\min(h(t),h_{max}(t)+1))\;,\; j=1,\ldots,m\;. \label{eq:loss_bound}
	\end{equation}
	\label{thm:mosoo}
\end{thm}
\noindent
\begin{proof}Since $|\mathcal{I}^h_j|\leq C_j{\delta_j(h)}^{-d_{s_j}}$ from Eq.~\eqref{eq:optimal_nodes_bound}; from the definition of $h(t)$~\eqref{eq:thm_mosoo}, we have:
\begin{equation}
h_{max}(t)\sum\limits_{l=0}^{h(t)-1}|\mathcal{I}^l_j|\leq h_{max}(t)\sum\limits_{l=0}^{h(t)-1}\max_{1\leq j \leq m}	C_j\delta_j(l)^{-d_{s_j}} < t \nonumber
\end{equation}
Thus, from Lemma \ref{lm:mosoo} and since $h_{max}(t)$ is the maximum depth at which nodes can be expanded, we have $\{h^*_{j,t}\}_{1 \leq j \leq m}\geq \min(h(t)-1,h_{max}(t))$.
Now, let $(h^*_{j,t}+1,j^*)$ be the deepest non-expanded $j$-optimal node (which is a child node of the deepest expanded $j$-optimal node at depth $h^*_{j,t}$ and its representative state~$\myvec{x}_{h^*_{j,t}+1,j^*}$ has been evaluated), then the loss with respect to the $j$th objective is bounded, based on Assumption~\ref{asmp:bounded_cell}, as:
\begin{equation}
r_j(t) \leq f(\myvec{x}_{h^*_{j,t}+1,j^*})-f(\myvec{x}_j^*) \leq \delta_j(h^*_{j,t}+1)\;.\nonumber
\end{equation}
Since $\{h^*_{j,t}\}_{1\leq j \leq m}\geq \min(h(t)-1,h_{max}(t))$, we have 	$r_j(t)\leq \delta_j(\min(h(t),h_{max}(t)+1))$, for $j=1,\ldots,m$.\end{proof}

\subsubsection{Bounding the Additive  Epsilon Indicator}
\label{sec:bounding_additive_indicator} Within the context of multi-objective optimization and after $t$ iterations, the vectorial loss $\myvec{r}(t)$ of Eq.~\eqref{eq:loss} does not explicitly capture the quality of \alg{MO-SOO}'s approximation set~$\mathcal{Y}^t_*$ with respect to the whole Pareto front~$\mathcal{Y}^*$. Here, we investigate whether there is an implicit connection between the two concepts. Particularly, we study the relationship between $\myvec{r}(t)$ (as well as its bound of Eq.~\ref{eq:loss_bound}) and the Pareto-compliant additive $\epsilon$-indicator of \alg{MO-SOO}'s approximation set~$\mathcal{Y}^t_*$ with respect to the Pareto front~$\mathcal{Y}^*$ (or the unary additive $\epsilon$-indicator of $\mathcal{Y}^t_*$): \IUnaryAddInd. In essence, \IUnaryAddInd~measures the smallest amount $\epsilon$ needed to translate each element in the Pareto front $\mathcal{Y}^*$ such that it is \emph{weakly} dominated by at least one element in the approximation set $\mathcal{Y}^t_*$. This notion is put formally in the next definition.

\begin{definition}
	(Additive $\epsilon$-indicator~\cite{zitzler2003performance}) For any two approximation sets $A, B \in \Omega$, the additive $\epsilon$-indicator $I_{\epsilon+}$ is defined as:
	\begin{equation}
	I_{\epsilon+}(A,B) = \inf_{\epsilon\in \mathbb{R}}\{\forall \myvec{y}^2 \in B,\; \exists \myvec{y}^1\in A : \myvec{y}^1 \preceq_{\epsilon+} \myvec{y}^2\}\,\label{eq:def_binary_epsilon}
	\end{equation}
	where $\myvec{y}^1 \preceq_{\epsilon+} \myvec{y}^2\iff y^1_j \leq \epsilon + y^2_j$ for all $j\in\{1,\ldots, m\}$. If $B$ is the Pareto front $\mathcal{Y}^*$ (or a good---in terms of diversity and closeness to the Pareto front---approximation reference set $R\in\Omega$ if $\mathcal{Y}^*$ is unknown) then $I_{\epsilon+}(A,B)$ is referred to as the unary additive  epsilon indicator and is denoted by $I^{1}_{\epsilon+}(A)$, \ie, $I^{1}_{\epsilon+}(A)\myeq I_{\epsilon+}(A,\mathcal{Y}^*)$.
	\label{def:epsilon_indicator}
\end{definition}

A negative value of $I_{\epsilon+}(A,B)$ indicates that $A$ strictly dominates $B$: every element in $B$ is strictly dominated by at least one element in $A$. Note that $I^{1}_{\epsilon+}(\mathcal{Y}^t_*)\myeq I_{\epsilon+}(\mathcal{Y}^t_*,\mathcal{Y}^*)\geq0$ as no element in $\mathcal{Y}^t_*$ strictly dominates any element in $\mathcal{Y}^*$. Thus, the closer \IUnaryAddInd~to $0$, the better the quality of $\mathcal{Y}^t_*$. Figure~\ref{fig:loss_vs_indicator} illustrates the two quantities, viz.~$\myvec{r}(t)$ and \IUnaryAddInd, and highlights their explicit relationship for $m=2$. From this observation, the following lemma is deduced.

\begin{figure}[htb]
	\begin{center}
		\renewcommand{\arraystretch}{1}
		\resizebox{0.8\textwidth}{!}{
			\begin{tabular}{@{}c@{}c@{}}
				\textsc{(i) Non-conflicting Objectives} & \textsc{(ii) Conflicting Objectives} \\
				\includegraphics[width=7cm, trim = 10mm 0mm 4mm 4mm, clip]{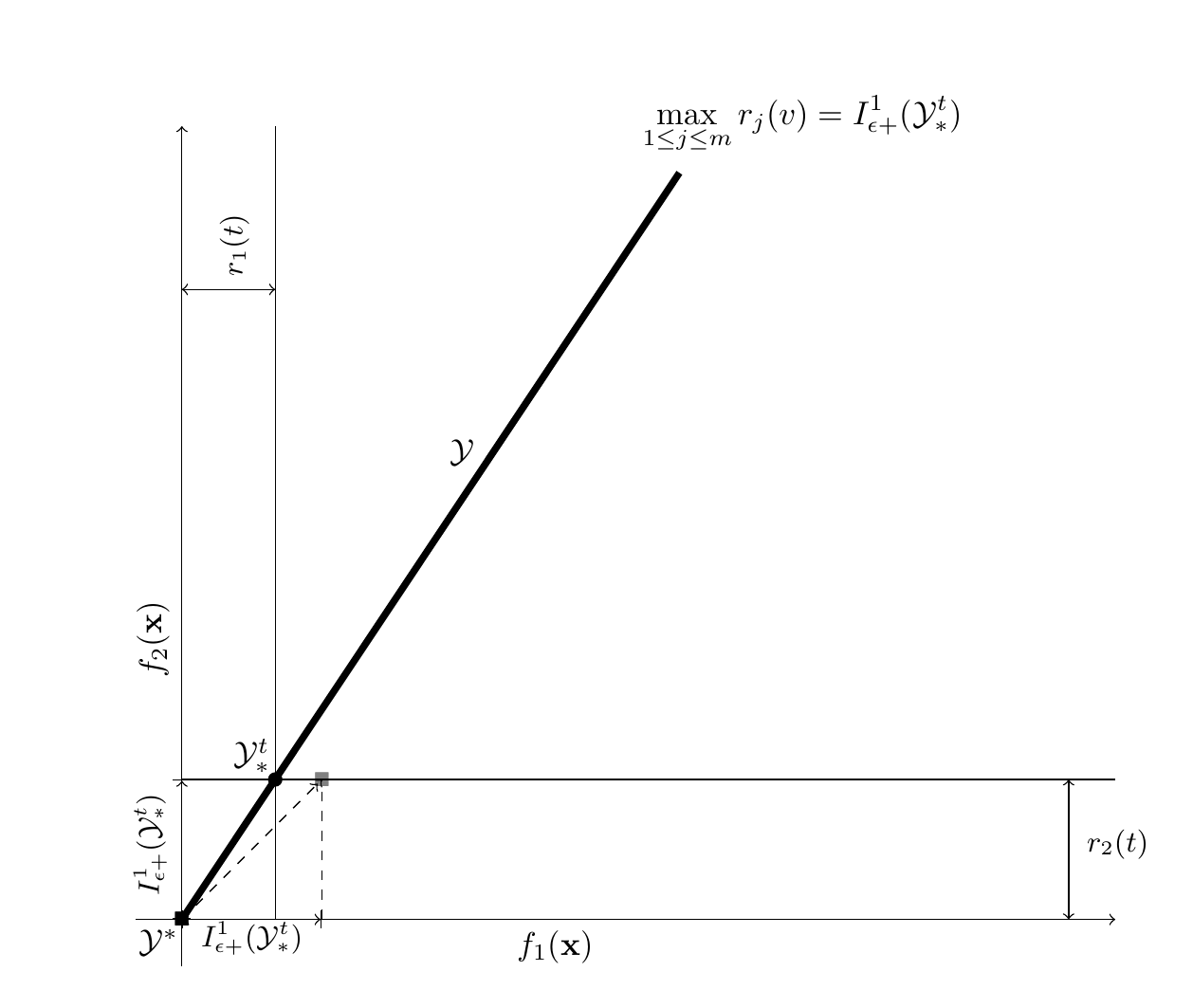} &
				\includegraphics[width=7cm, trim = 10mm 0mm 4mm 4mm, clip]{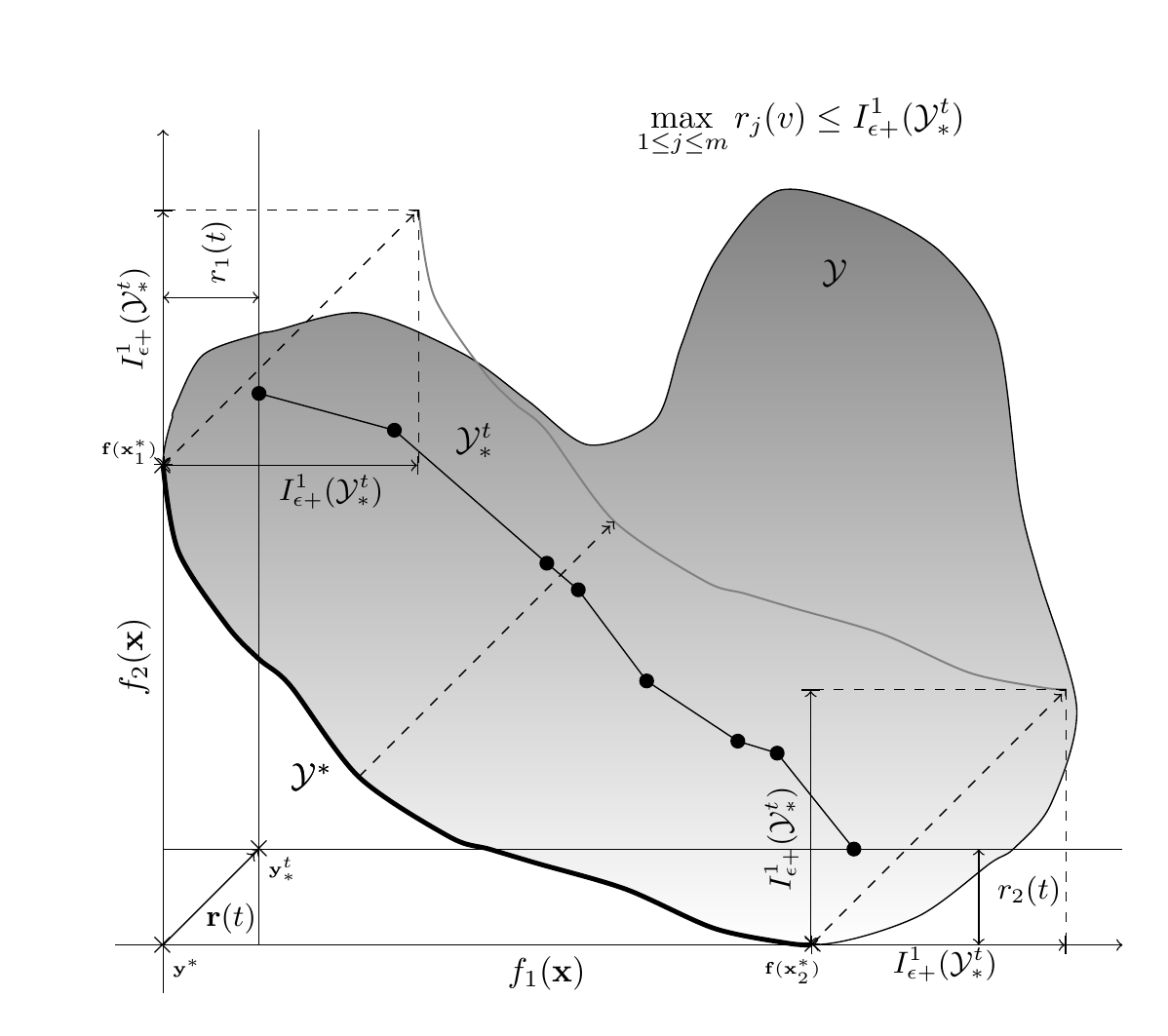} \\
		\end{tabular}}
	\end{center}
	\caption{Illustration of the vectorial loss $\mathbf{r}(t)$ of Eq.~\eqref{eq:loss} and its relation to the unary additive epsilon indicator $I^{1}_{\epsilon+}(\mathcal{Y}^t_*)$~(Definition~\ref{def:epsilon_indicator}) for a multi-objective problem ($m=2$) whose objective space~$\mathcal{Y}$ is shown in two scenarios: (i)~non-conflicting objectives and (ii)~conflicting objectives. The gray square (resp., curve) in the first (resp., second) scenario represents the least-translated Pareto front so as every translated element is  weakly dominated by at least one element in the approximation set $\mathcal{Y}^t_*$,  \ie, $\{(y_1+I^{1}_{\epsilon+}(\mathcal{Y}^t_*),\ldots,y_m+I^{1}_{\epsilon+}(\mathcal{Y}^t_*))\}_{\myvec{y}\in\mathcal{Y}^*}$. Mathematically, $I^{1}_{\epsilon+}(\mathcal{Y}^t_*)\geq\max_{1\leq j\leq m}{r}_{j}(t)$ where equality sufficiently holds when $|\mathcal{Y}^t_*|=1$~(see Lemma~\ref{lm:lower_bounding_epsilon}).}
	\label{fig:loss_vs_indicator}		
\end{figure}

\begin{figure}[htb]
	\begin{center}
		\renewcommand{\arraystretch}{1.2}
		\resizebox{0.65\textwidth}{!}{
			\includegraphics[scale=0.65,trim=5 10 2 10 mm, clip=true]{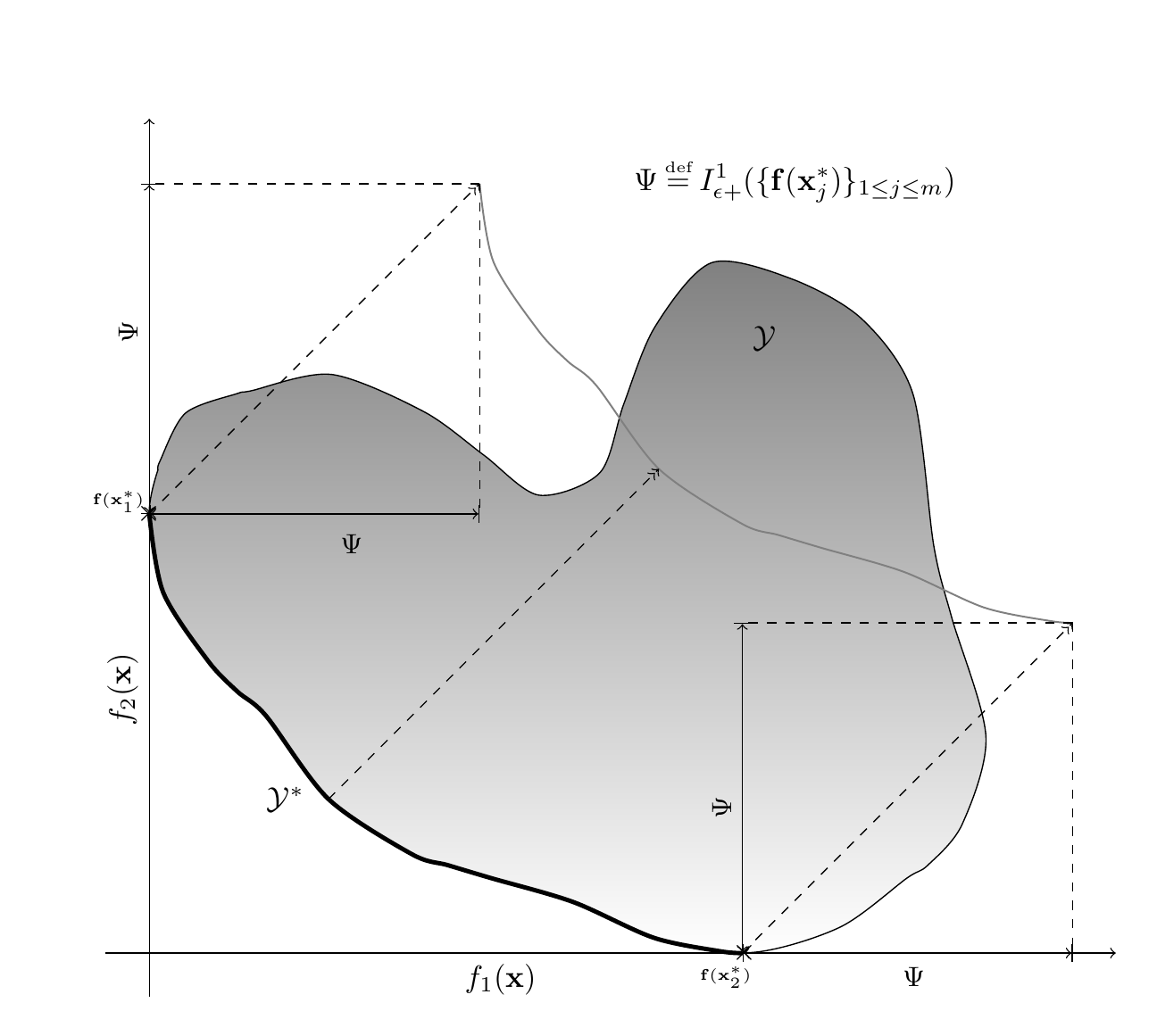}}
	\end{center}
	\caption{The objective space~$\mathcal{Y}$ for a multi-objective optimization problem ($m=2$) with conflict dimension~$\Psi\geq 0$ (Definition~\ref{def:conflict_dimension}). The solid curve marks the Pareto front $\mathcal{Y}^*$. The faded curve represents the $\Psi$-translated Pareto front, \ie, $\{(y_1+\Psi,\ldots,y_m+\Psi)\}_{\myvec{y}\in\mathcal{Y}^*}$. Every element of the $\Psi$-translated Pareto front is  \emph{weakly} dominated by at least one element in the set $\{\mathbf{f}(\mathbf{x}^*_j)\}_{1\leq j \leq m} \subseteq \mathcal{Y}^*$.}
	\label{fig:conflict_dimension}		
\end{figure}

\begin{lem}
	For any MOO solver, we have \IUnaryAddInd$\geq \max_{1\leq j\leq m}r_j(t)$.
	\label{lm:lower_bounding_epsilon}
\end{lem}

\begin{proof}From the definition of the vectorial loss measure~\eqref{eq:loss}, the $m$ closest elements on the approximation set $\mathcal{Y}^t_*$ to the $m$ extrema of the Pareto front~$\mathcal{Y}^*$---\ie, $\{\myvec{f}(\myvec{x}^*_j) \}_{1\leq j\leq m}$---differ by $\{r_j(t)\}_{1\leq j \leq m}$ along the corresponding $j$th objective, respectively. Therefore, an objective-wise translation  of at least $\max_{1\leq j \leq m}r_j(t)$ is needed so as each of the translated Pareto front extrema is \emph{weakly} dominated by at least one element in the approximation set~$\mathcal{Y}^t_*$. Thus, from Definition~\ref{def:epsilon_indicator}, \IUnaryAddInd~$\geq\max_{1\leq j \leq m}r_{j}(t)$.\end{proof}

While Lemma~\ref{lm:lower_bounding_epsilon} provides a lower bound on the indicator~\IUnaryAddInd, one is more interested in an upper bound so as to capture the convergence of the approximation set to the whole Pareto front. To this end, we propose a measure of conflict of the Pareto front extrema with respect to the rest of its elements, called \emph{conflict dimension}.

\begin{definition}(conflict dimension) The conflict dimension~$\Psi\geq 0$ for an MOO problem with $m$ objectives and Pareto front~$\mathcal{Y}^*$ is the unary additive epsilon indicator of the approximation set that consists of the extrema of $\mathcal{Y}^*$ ($m$ or less elements). Mathematically:
	\begin{equation}
	\Psi \myeq I^{1}_{\epsilon+}(\{\myvec{f}(\myvec{x}^*_j)\}_{1\leq j \leq m}) \myeq I_{\epsilon+}(\{\myvec{f}(\myvec{x}^*_j)\}_{1\leq j \leq m}, \mathcal{Y}^*)\label{eq:conflict_dimensionality}
	\end{equation}
	\label{def:conflict_dimension}
\end{definition}


Figure~\ref{fig:conflict_dimension} illustrates the proposed measure. Note that $\Psi$ is an intrinsic property of the MOO problem's Pareto front~$\mathcal{Y}^*$. In essence, the conflict dimension~$\Psi$ captures the proximity of Pareto front extrema to the rest of its elements, where~$\Psi=0 \iff |\mathcal{Y}^*|\leq m$. We now provide our upper bound on the indicator~\IUnaryAddInd.

\begin{thm}[\IUnaryAddInd~for \alg{MO-SOO}] Let us define $\acute{h}(t)\myeq\min(h(t),h_{max}(t)+1)$ where $t$ is the number of iterations and $h(t)$---as in Theorem~\ref{thm:mosoo}---is the smallest $h\geq0$ such that Eq.~\eqref{eq:thm_mosoo} holds. Then for an MOO problem with conflict dimension~$\Psi$, the indicator~\IUnaryAddInd~of \alg{MO-SOO} is bounded as:
	\begin{equation}
	I^{1}_{\epsilon+}(\mathcal{Y}^t_*) < {\Psi+\max_{1\leq k,l \leq m} (1 + 2C_k{\delta_k(\acute{h}(t))}^{-d_{s_k}})\cdot\delta_l(\acute{h}(t))}\;.\label{eq:indicator_bound}
	\end{equation}
	\label{thm:indicator_mosoo}
\end{thm}
\noindent
\begin{proof}  From the loss bound~\eqref{eq:loss_bound} established in Theorem~\ref{thm:mosoo}, \alg{MO-SOO}'s approximation set after $t$ iterations~$\mathcal{Y}^t_*$ lies in a portion of the function space, namely $\{\myvec{f}(\mathcal{X}_1^{\delta_1(\acute{h}(t))}),\ldots,\myvec{f}(\mathcal{X}_m^{\delta_m(\acute{h}(t))})\}$ and possibly $\myvec{f}(\mathcal{X}^{\acute{h}(t)}_{ND})$ (defined before Lemma~\ref{lm:mosoo} in Section~\ref{sec:bounding_loss} and depicted in Figure~\ref{fig:conv_preliminary_1}, for $m=2$). Therefore, in the worst-case scenario, $\mathcal{Y}^t_*$ consists of $m$ (or less) elements that contribute to the nadir point of $\myvec{f}(\mathcal{X}^{\acute{h}(t)}_{ND})$ (\eg, $\mathcal{Y}^t_*=\{\myvec{y}^1,\myvec{y}^2\}$ in Figure~\ref{fig:conv_preliminary_1}, where their objective-wise values constitute $\myvec{y}^3$). For brevity, let us denote this worst-case approximation set and the set of the Pareto front extrema $\{\myvec{f}(\myvec{x}^*_j)\}_{1\leq j \leq m}$ by $\mathcal{Y}^{t,w}_*$ and $\mathcal{Y}^{*,e}$, respectively.

Now, for all $j\in\{1,\ldots,m\}$, the maximum objective-wise translation between the element $\mathcal{Y}^{t,w}_*\cap \myvec{f}(\mathcal{X}_j^{\delta_j(\acute{h}(t))})$ (\eg, $\myvec{y}^1$ in Figure~\ref{fig:conv_preliminary_1} for $j=1$) and $\myvec{f}(\myvec{x}^*_j)\in\mathcal{Y}^{*,e}$ is upper bounded as follows (see Figure~\ref{fig:conv_preliminary_1} for illustration).
\begin{eqnarray}
\max_{\substack{1\leq \bar{\jmath}\leq m\\\myvec{y}^1\in \mathcal{Y}^{t,w}_* \cap \myvec{f}(\mathcal{X}_j^{\delta_j(\acute{h}(t))})}}
y^1_{\bar{\jmath}} - f_{\bar{\jmath}}(\myvec{x}^*_j)&\leq& \max\Big(\max_{\substack{1\leq l\leq m\\ l\neq j}}\overbrace{ C_j{\delta_j(\acute{h}(t))}^{-d_{s_j}}}^\text{from Eq.~\eqref{eq:optimal_nodes_bound}}\cdot \underbrace{2\delta_l(\acute{h}(t))}_\text{from~\ref{asmp:smoothness} and~\ref{asmp:bounded_cell}}, \overbrace{\delta_j(\acute{h}(t))}^\text{from Eq.~\eqref{eq:loss_bound}}\Big) \nonumber\\
&\leq&\max_{1\leq k\leq m} \Bigg( \max\Big(\max_{\substack{1\leq l\leq m\\ l\neq k}} 2C_k{\delta_k(\acute{h}(t))}^{-d_{s_k}}\delta_l(\acute{h}(t)), \delta_k(\acute{h}(t)) \Big)\Bigg)\nonumber\\
&<&  \max_{1\leq k,l \leq m} (1 + 2C_k{\delta_k(\acute{h}(t))}^{-d_{s_k}})\cdot\delta_l(\acute{h}(t))\;.\label{eq:worst_extreme_translation}
\end{eqnarray}
Put it differently, elements in $\mathcal{Y}^{t,w}_*$ differ objective-wise by a value less than the right-hand side of~\eqref{eq:worst_extreme_translation} with respect to their corresponding closest elements in $\mathcal{Y}^e_*$.
On the other hand, Definition~\ref{def:conflict_dimension} implies that there exists at least one element $\myvec{y}^2\in\mathcal{Y}^{*,e}$ for every element $\myvec{y}^3\in \mathcal{Y}^*$ such that $\myvec{y}^2 \preceq_{\Psi}\myvec{y}^3$. Consequently, we have
\begin{multline}
{y}^2_j+ \max_{1\leq k,l \leq m} (1 + 2C_k{\delta_k(\acute{h}(t))}^{-d_{s_k}})\cdot\delta_l(\acute{h}(t))\leq\\ {\Psi+\max_{1\leq k,l \leq m} (1 + 2C_k{\delta_k(\acute{h}(t))}^{-d_{s_k}})\cdot\delta_l(\acute{h}(t))} + {y}^3_j\label{eq:extreme_translation}\;,
\end{multline}
for all $j\in\{1,\ldots,m\}$. Combining~\eqref{eq:worst_extreme_translation} and~\eqref{eq:extreme_translation} indicates that for every element $\myvec{y}^3\in\mathcal{Y}^*$, there exists at least one element $\myvec{y}^1\in\mathcal{Y}^{t,w}_*$ such that
\begin{equation}
{y}^1_j< {\Psi+\max_{1\leq k,l \leq m} (1 + 2C_k{\delta_k(\acute{h}(t))}^{-d_{s_k}})\cdot\delta_l(\acute{h}(t))} + {y}^3_j\label{eq:pareto_worst_translation}\;,
\end{equation}
for all $j\in\{1,\ldots,m\}$. In other words, $$I^{1}_{\epsilon+}(\mathcal{Y}^{t,w}_*)\leq {\Psi+\max_{1\leq k,l \leq m} (1 + 2C_k{\delta_k(\acute{h}(t))}^{-d_{s_k}})\delta_l(\acute{h}(t))}$$. Since $\max_{\myvec{y}^4\in\mathcal{Y}^t_*} y^4_j \leq\max_{\myvec{y}^1\in\mathcal{Y}^{t,w}_*} y^1_j$ for all $j\in\{1,\ldots,m\}$, we have
\begin{equation}
I^{1}_{\epsilon+}(\mathcal{Y}^{t}_*)< {\Psi+\max_{1\leq k,l \leq m} (1 + 2C_k{\delta_k(\acute{h}(t))}^{-d_{s_k}})\cdot\delta_l(\acute{h}(t))}\nonumber\;.
\end{equation}
\end{proof}

Theorem~\ref{thm:indicator_mosoo} describes the bound on~\IUnaryAddInd~by a non-increasing function, viz. $\max_{1\leq k,l \leq m}(1 + 2C_k{\delta_k(\acute{h}(t))}^{-d_{s_k}})\cdot\delta_l(\acute{h}(t))$ reflecting the objectives smoothness with an offset dependent on the structure of the Pareto front~with regard to its extrema~$\{\myvec{f}(\myvec{x}^*_j)\}_{1\leq j \leq m}$---\ie, the conflict dimensionality~$\Psi$.
Section~\ref{sec:illust} gives some illustrative examples about the characteristics of the non-increasing function in relation to the theoretical bounds presented. In Section~\ref{sec:validating_bounds}, these theoretical bounds are calculated via \emph{symbolic computation} and validated on a set of synthetic problems. It should be noted that as the rightmost term---of the Eq.~\ref{eq:thm_mosoo}---diminishes to zero, the presented upper-bound holds down to $\Psi$ and fails to characterize/follow \IUnaryAddInd~afterwards. This does not imply that \IUnaryAddInd~will not decrease henceforth but rather does not guarantee the same. Nevertheless, the result of the next section indicates that in the limit \IUnaryAddInd~decreases to zero, since \alg{MO-SOO}'s approximation set converges asymptotically to the Pareto front, as supported by the empirical validation of Section~\ref{sec:validating_bounds}.

\subsection{Asymptotic Performance}
\label{sec:asymptotic}
Theorem~\ref{thm:mosoo} addressed the finite-time performance of \alg{MO-SOO} with respect to $m$ points on the Pareto front, whereas Theorem~\ref{thm:indicator_mosoo} established it with respect to the additive $\epsilon$-indicator as the number of iterations~$t$ grows. Here, we consider the asymptotic behavior of \alg{MO-SOO}, that is, its approximation set given an infinite budget of function evaluations.  Asymptotic analysis has been the core of convergence studies of several established algorithms (see, e.g., ~\cite{Torn1989,Conn1997,Sergeyev1998,Kelley1999,Huyer1999,Lewis2002,Finkel2004,Audet2006,Conn2009,Sergeyev2016}).
In this section, we show that \alg{MO-SOO} asymptotically converges to the whole Pareto front.

\alg{MO-SOO} guarantees that no portion of $\mathcal{X}$ is disregarded $\iff$ $h_{max}(t)\to \infty$ as $t\to\infty$. Accordingly, if a Pareto optimal node happens to be a leaf node at iteration $\acute{t}$, then it will definitely get expanded in one of the next iterations $\geq \acute{t}+1$. As the number of iterations $t$ grows bigger, the base points sampled by \alg{MO-SOO} form a dense subset of $\mathcal{X}$ such that for an arbitrary small $\epsilon\geq 0$: 
$ \forall \acute{\myvec{x}}\in \mathcal{X}^*, \exists\text{ a base point }\myvec{x} \text{ such that } |\myvec{f}(\myvec{x})-\myvec{f}(\acute{\myvec{x}})|\leq \epsilon$. The next theorem establishes formally our proposition about the consistency property of \alg{MO-SOO}.

\begin{thm}[\alg{MO-SOO} Consistency]
	\alg{MO-SOO} is consistent, if $h_{max}(t)\to \infty$ as $t\to\infty$, where $t$ is the number of iterations.\label{thm:asymp}
\end{thm}
\begin{proof} Let us denote the deepest Pareto optimal node that has the Pareto optimal solution $\acute{\myvec{x}}\in \mathcal{X}^*$ by $(h_{\acute{\myvec{x}}}(t),i_{\acute{\myvec{x}}})$. \ie, $\acute{\myvec{x}}\in \mathcal{X}_{h_{\acute{\myvec{x}}}(t),i_{\acute{\myvec{x}}}}$.  From Assumption~\ref{asmp:bounded_cell} and the definition of the semi-metric~$\ell_j$,
\begin{align*}
0\leq \ell_j(\myvec{x}_{h_{\acute{\myvec{x}}}(t),i_{\acute{\myvec{x}}}},\acute{\myvec{x}}) &\leq\;\; \delta_j(h_{\acute{\myvec{x}}}(t))\;\;\;, \forall \acute{\myvec{x}}\in \mathcal{X}^*,j=1,\ldots,m\;.\\
\intertext{{Since $h_{max}(t)\to \infty$ as $t\to\infty$, the depths of all the Pareto optimal nodes tends to $\infty$, mathematically:}}
0\leq \lim_{t\to\infty}\ell_j(\myvec{x}_{h_{\acute{\myvec{x}}}(t),i_{\acute{\myvec{x}}}},\acute{\myvec{x}}) &\leq \lim_{h_{\acute{\myvec{x}}}(t)\to\infty}\delta_j(h_{\acute{\myvec{x}}}(t))\;\;\;, \forall \acute{\myvec{x}}\in \mathcal{X}^*,j=1,\ldots,m\;. \\
\intertext{\text{Then, with Assumption~\ref{asmp:bounded_cell}:}}
\lim_{t\to\infty}\ell_j(\myvec{x}_{h_{\acute{\myvec{x}}}(t),i_{\acute{\myvec{x}}}},\acute{\myvec{x}})&= \;\;\;\;0\;\;\;\;\;\;\;\;\;\;\;\;\;\;\;\;\;\;\;\;\;\;\;\;, \forall \acute{\myvec{x}}\in \mathcal{X}^*,j=1,\ldots,m\;, \\
\intertext{and from the coincidence axiom satisfied by $\ell_j$ as a semi-metric:}
\lim_{t\to\infty}\myvec{x}_{h_{\acute{\myvec{x}}}(t),i_{\acute{\myvec{x}}}} &=\;\;\;\;\acute{\myvec{x}}\;\;\;\;\;\;\;\;\;\;\;\;\;\;\;\;\;\;\;\;\;\;\;\;\;, \forall \acute{\myvec{x}}\in \mathcal{X}^*\;.
\end{align*}
Thus, as the number of iterations $t$ grows bigger, \alg{MO-SOO} asymptotically converges to the Pareto front.\end{proof}



\subsection{Illustration}
\label{sec:illust}
In this section, insights on the loss bound~\eqref{eq:loss_bound} is presented and illustrated through some examples.\footnote{As the indicator bound~\eqref{eq:indicator_bound} is dependent on the loss bound~\eqref{eq:loss_bound}, similar analysis holds true for the indicator bound as well.} For $j=1,\ldots,m$; let $\delta_j(h)=c_j\gamma_j^h$ for some constants $c_j>0$ and $0<\gamma_j<1$; $h_{max}(t)=t^{p}$ for $p\in(0,1)$. Putting this in \eqref{eq:loss_bound}, two interesting cases can be noted:
\begin{itemize}
	\item Consider the case where  $\{d_{s_j}\}_{0\leq j \leq m}=0$, denote $\max_{1\leq j \leq m}	C_j$ by $\hat{C}_j$. From Theorem~\ref{thm:mosoo}:
	\begin{eqnarray}
	t\leq h_{max}(t)\sum\limits_{l=0}^{h(t)}\max_{1\leq j \leq m}	C_j\delta_j(l)^{-d_{s_j}} = 	h_{max}(t)\cdot\hat{C}_j(h(t)+1)\nonumber\;.
	\end{eqnarray}
	Thus, for $j=1,\ldots,m$: 
	\begin{equation}
	r_j(t)\leq {O}(\gamma_j^{\min(t^{1-p},\;t^p)})\;,\label{eq:exp_loss}
	\end{equation}
	\ie, the loss is a stretched-exponential function of the number of iterations $t$.
	\item Consider the case where $\exists k\in\{1,\ldots,m\}$ such that $$\forall l,\;C_k\delta_k(l)^{-d_{s_k}} = \max_{1\leq j \leq m}	C_j\delta_j(l)^{-d_{s_j}}$$ and $d_{s_k}>0$, then from Theorem~\ref{thm:mosoo}, we have:
	\begin{eqnarray}
	t\leq h_{max}(t)\sum\limits_{l=0}^{h(t)}C_k\delta_k(l)^{-d_{s_k}} &=& 	C_k\cdot c^{-d_{s_k}}_{k}\cdot t^p\cdot\frac{\gamma^{-d_{s_k}(h(t)+1)}_k-1}{\gamma^{-d_{s_k}}_k-1}\nonumber\;,\\
	\frac{(1-\gamma^{d_{s_k}}_k)}{C_k}\cdot t^{1-p} &\leq &c^{-d_{s_k}}_k\gamma^{-d_{s_k}h(t)}_k\;,\nonumber\\
	\Bigg(\frac{C_k}{1-\gamma^{d_{s_k}}}\Bigg)^{1/d_{s_k}}\cdot t^{-\frac{1-p}{d_{s_k}}} &\geq &  c_k\gamma^{h(t)}_k\;.\nonumber
	\end{eqnarray}
	Hence, $h(t)$ is of a logarithmic order in $t$, making $h(t)< h_{max}(t) + 1$ as $t$ grows bigger.
	Thus, with $\delta_j(h)=\Theta(\delta_k(h))$ for $j=1,\ldots,m$;
	\begin{equation}
	r_j(t)\leq {O}(t^{-\frac{1-p}{d_{s_k}}})\;,\label{eq:poly_loss}
	\end{equation}
	\ie, the loss is a polynomially-decreasing function of the number of iterations~$t$.
\end{itemize}
One can deduce that the performance (in terms of the loss~\eqref{eq:loss}) is influenced by two main factors, viz. the near-optimality dimension of the objectives $\{d_{s_j}\}_{0\leq j \leq m}$ , and the maximal depth function~$h_{max}(t)$.

\textit{The Maximal Depth Function $h_{max}(t)$}.
From Theorem~\ref{thm:mosoo}, the maximal depth function $h_{max}(t)$ acts as a multiplicative factor in the definition of $h(t)$ (Eq.~\ref{eq:thm_mosoo}) as well as a limiting factor on the loss bound (Eq.~\ref{eq:loss_bound}). This effect of $h_{max}(t)$ elegantly captures the exploration-vs.-exploitation trade-off. Larger $h_{max}(t)$ makes the algorithm more exploitative (deeper tree) and $h(t)$ smaller, while smaller $h_{max}(t)$ makes the algorithm more exploratory (broader tree) and $h(t)$ larger; the inverse proportionality between $h_{max}(t)$ and $h(t)$ evens out the loss bound in both situations.

\begin{figure}[tb]
	\begin{center}
		\renewcommand{\arraystretch}{1.2}
		\includegraphics[scale=0.8,trim=5 50 2 50 mm, clip=true]{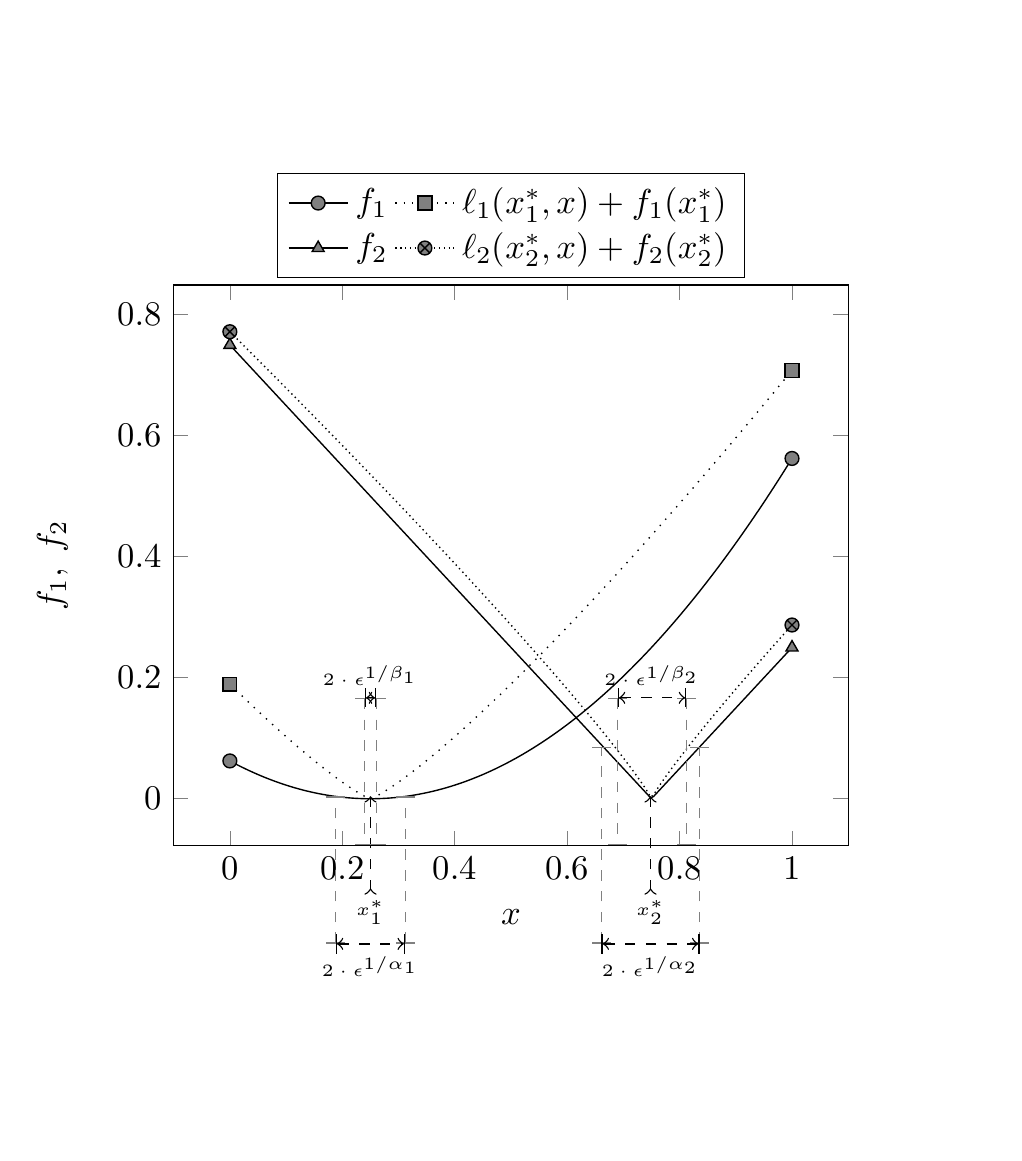}
	\end{center}
	\caption{Bi-objective problem ($m=2,n=1$) over $\mathcal{X}=[0,1]$ with $f_1(\myvec{x})=||\myvec{x}-0.25||^{\alpha_1}_\infty$, $f_2=||\myvec{x}-0.75||^{\alpha_2}_\infty$, $\ell_1(\myvec{x},\myvec{y})=||\myvec{x}-\myvec{y}||^{\beta_1}_{\infty}$, $\ell_2(\myvec{x},\myvec{y})=||\myvec{x}-\myvec{y}||^{\beta_2}_{\infty}$ where $\beta_1\leq \alpha_1$, $\beta_2\leq \alpha_2$. The region $\mathcal{X}^{\epsilon}_1$ (resp., $\mathcal{X}^{\epsilon}_2$) is the interval centered around $x^*_1$ (resp., $x^2_*$) of length $2\cdot\epsilon^{1/\alpha_1}$ (resp., $2\cdot\epsilon^{1/\alpha_2}$). They can be packed with $\epsilon^{1/\alpha_1-1/\beta_1}$ (resp., $\epsilon^{1/\alpha_2-1/\beta_2}$) intervals of length $2\cdot\epsilon^{1/\beta_1}$ (resp., $2\cdot\epsilon^{1/\beta_2}$).}
	\label{fig:bi_obj_ill}		
\end{figure}

\textit{The Near-Optimality Dimensions $\{d_{s_j}\}_{0\leq j \leq m}$}.
While $h_{max}(t)$ is a parameter of the algorithm, $\{d_{s_j}\}_{0\leq j \leq m}$ 
are dependent on the multi-objective problem at hand and are related to the algorithm's partitioning strategy through the scaling factors $\{s_j\}_{1\leq j \leq m}$. Consider the near-optimality dimensions for the bi-objective problem (depicted in Figure~\ref{fig:bi_obj_ill} for $n=1$) where $\mathcal{X}=[0,1]^n$, $f_1(\myvec{x})=||\myvec{x}-0.25||^{\alpha_1}_\infty$, and $f_2=||\myvec{x}-0.75||^{\alpha_2}_\infty$ for $\alpha_1\geq 1$, $\alpha_2\geq1$; and let \alg{MO-SOO} have a partition factor of $K=3^n$. Furthermore, assume the semi-metrics to be $\ell_1(\myvec{x},\myvec{y})=||\myvec{x}-\myvec{y}||^{\beta_1}_{\infty}$, $\ell_2(\myvec{x},\myvec{y})=||\myvec{x}-\myvec{y}||^{\beta_2}_{\infty}$ where $\beta_1\leq \alpha_1$, $\beta_2\leq \alpha_2$ in line with Assumption~\ref{asmp:smoothness}. In the light of Assumption~\ref{asmp:bounded_cell}, $\delta_1(h)$ and $\delta_2(h)$ may be written as $2^{-\beta_1}\cdot 3^{-h\beta_1}$ and $2^{-\beta_2}\cdot 3^{-h\beta_2}$, respectively; and from Assumption~\ref{asmp:shaped_cells}, we have $s_1=1$ and $s_2=1$. The region $\mathcal{X}^{\delta_1(h)}_1$ (resp., $\mathcal{X}^{\delta_2(h)}_2$) is the $L_\infty$-ball of radius $\delta_1(h)^{1/\alpha_1}$ (resp., $\delta_2(h)^{1/\alpha_2}$) centered in $\myvec{0.25}$ (resp., $\myvec{0.75}$). In line of Definition~\ref{def:near_opt_dim}, these regions can be packed by $\Big(\frac{{\delta_1(h)}^{1/\alpha_1}}{{\delta_1(h)}^{1/\beta_1}}\Big)^n$ (resp., $\Big(\frac{{\delta_2(h)}^{1/\alpha_2}}{{\delta_2(h)}^{1/\beta_2}}\Big)^n$) $L_\infty$-balls of radius $\delta_1(h)^{1/\beta_1}$ (resp., $\delta_2(h)^{1/\beta_2}$). Thus the near-optimality dimensions are $d_{s_1}=n(1/\beta_1-1/\alpha_1)$ and $d_{s_2}=n(1/\beta_2-1/\alpha_2)$. \textit{Without loss of generality}, three scenarios are present with respect to the first objective:
\begin{enumerate}
	\item $\alpha_1= \beta_1 \implies d_{s_1}=0$; the cardinality of the set $\mathcal{I}^h_1$ is a constant regardless of the depth~$h$ and the decision space dimensionality~$n$. This presents a balanced trade-off between exploration and exploitation as the semi-metric~$\ell_1$ is capturing the function~$f_1$ behavior precisely.
	\item $\alpha_1> \beta_1 \implies d_{s_1}>0\;$; the cardinality of the set $\mathcal{I}^h_1$ becomes an increasing function of the depth~$h$ and the decision space dimensionality~$n$. This presents a bias towards exploration as the semi-metric $\ell_1$ underestimates the behavior of the function $f_1$.
	\item $\alpha_1< \beta_1$; this violates Assumption~\ref{asmp:smoothness}. With this regards, the algorithm becomes more exploitative falling for local optimal solutions as the semi-metric~$\ell_1$ is overestimating $f_1$'s smoothness.
\end{enumerate}
The first two scenarios coincide with the two cases discussed earlier in this section. As $t$ grows larger and the near-optimality dimensions are zero (reflecting a balance in the exploration-vs.-exploitation dilemma), setting $p=0.5$ in $h_{max}(t)=t^{p}$ results in a faster decay of the loss bound (Eq.~\ref{eq:exp_loss}). On the other hand, when more exploration is needed, setting $p\to0$ (broader tree) gives a faster loss bound decay (Eq.~\ref{eq:poly_loss}). 


\begin{remark} It is important to reiterate here that \alg{MO-SOO} does not need the knowledge of the functions smoothness and the corresponding near-optimality dimensions, it only requires the \emph{existence} of such smoothness. These measures help only in quantifying the algorithm's performance.
\end{remark}

\begin{remark} The case of zero near-optimality dimension covers a large class of functions. In fact, it has been shown by \cite{valkostochastic} that the near-optimality dimension is zero for any function defined over a finite-dimensional and bounded space, and  whose upper- and lower-envelopes around the global optimizer are of the same order.
\end{remark}

\subsection{Empirical Validation of Theoretical Bounds}
\label{sec:validating_bounds}

In this section, the loss $\myvec{r}(t)$ and the indicator~\IUnaryAddInd~bounds of~\eqref{eq:loss_bound} and~\eqref{eq:indicator_bound}, respectively, are validated empirically for the bi-objective problem defined in Section~\ref{sec:illust} and depicted in Figure~\ref{fig:bi_obj_ill}. We compute these quantities using the \emph{Symbolic Math Toolbox} from \textsf{The MathWorks, Inc.} and compare them with respect to the numerical loss and indicator values obtained by running \alg{MO-SOO} with an evaluation budget of $v=10^4$ function evaluations.

With a partition factor of $K=3$, the decreasing sequence $\delta_1(h)$ (resp., $\delta_2(h)$) can be defined as $2^{-\alpha_1}\cdot K^{-3\alpha_1 \lfloor{h/n}\rfloor}$ (resp., $2^{-\alpha_2}\cdot K^{-3\alpha_2 \lfloor{h/n}\rfloor}$) as the search space is partitioned coordinate-wise per depth. Moreover, from Assumption~\ref{asmp:shaped_cells}, we have $s_1=1$ (resp., $s_2=1$). $C_1$ and $C_2$ of Definition~\ref{def:near_opt_dim} are set to $2$ as the cell centers may lie on the boundary of $\mathcal{X}^{\delta_1(h)}_1$ and $\mathcal{X}^{\delta_2(h)}_2$, respectively. 

To assess the effect of the conflict dimension $\Psi$ (defined in Definition~\ref{def:conflict_dimension}), eight instances of the problem are tested, where $n\in\{1,2\}$, and the $j$-optimal solutions ($x^*_1,x^*_2$) are set in one of four configurations---reflecting among others the maximum and minimum $\Psi$ values. The Pareto front $\mathcal{Y}^*$ and the conflict dimension $\Psi$ of the problem are estimated numerically from $10^6$~uniformly-sampled points. While the maximal depth function $h_{max}(t)$ acts as a very conservative multiplicative factor in~\eqref{eq:thm_mosoo} for the number of depths visited in each iteration. In our experiments, we have recorded the number of depths visited in each iteration and used the recorded values as the multiplicative factor in computing the theoretical bounds of~\eqref{eq:loss_bound} and~\eqref{eq:indicator_bound}.

The numerical and theoretical measures are presented in Figure~\ref{fig:bounds_validation}. First, one can easily verify Lemma~\ref{lm:lower_bounding_epsilon}. Second,
whilst having the same evaluation budget~$v$, the conflict and decision space dimensions have a clear impact on the corresponding number of iterations~$t$. Recall that one iteration represents executing the lines~\ref{ln:beg_iteration}--\ref{ln:end_iteration} of Algorithm~\ref{alg:mosoo}, once. Though with some offset, one can note how the theoretical measures upper bound the numerical measures with a similar trend. The code for generating the data presented in this section is available at \url{https://www.dropbox.com/s/ssiq1m52hczuj7a/mosoo-theory-validation.rar?dl=0}.

\section{Experimental Assessment}
\label{sec:assessment}

Due to space limitations,  the experimental validation of \alg{MO-SOO} and its comparison with several state-of-the-art algorithms is presented in detail in the online supplement, which is available at ~\url{https://www.dropbox.com/s/lifnnz0ajzjxdks/mosoo-supplement-quantiles.pdf?dl=0}.

\section{Conclusion}
\label{sec:conclusion}

This paper presents the Multi-Objective Simultaneous Optimistic Optimization (\alg{MO-SOO}): an optimistic approach to solve multi-objective optimization problems given a finite number of function evaluations. Using a tree of bandits, \alg{MO-SOO} hierarchically partitions the feasible decision space in search for Pareto optimal solutions using the non-dominated Pareto relation among its tree nodes. \alg{MO-SOO} performance  in terms of finite-time rate as well as asymptotic convergence has been studied, based on three basic assumptions about the function smoothness and hierarchical partitioning.
While existing theoretical analysis of MOO solvers either considers finite-set/discrete problems, provides probabilistic guarantees, or asymptotic local stationarity convergence, 
the theoretical analysis of \alg{MO-SOO} establishes a deterministic upper bound on the Pareto-compliant  $\epsilon$-indicator for continuous MOO problems that holds down to a problem-dependent measure, namely the conflict dimension, which captures the structure of the problem's Pareto front with respect to its extrema. Furthermore, it has been shown that \alg{MO-SOO} converges asymptotically to the Pareto front.

The empirical performance of \alg{MO-SOO} in approximating Pareto fronts has been evaluated using 300 benchmark MOO problems and their results are compared with three state-of-the art MOO solvers, namely  \alg{MOEA/D}, \alg{MO-CMA-ES}, and \alg{SMS-EMOA}.  The performance of \alg{MO-SOO} is comparable with best results of the top performing \alg{SMS-EMOA} algorithm. From results, we observe that problems with weakly-structured multi-modal objectives impose a challenge for \alg{MO-SOO}. This can be attributed to two factors: theoretical foundation of the algorithm (the near-optimality dimension) in scaling the exploration proportionally with the number of objective-wise global optima and the fact that sequential partitioning scheme may not adapt well in case of weakly-structured objectives. In addition, the nature of the used $\ND(\cdot)$ operator overlooks the diversity of the selected nodes for expansion.

\section*{Acknowledgement}
The authors wish to thank the ATMRI:2014-R8, Singapore, for providing financial support  to conduct this study. Thanks extended to 
Dimo Brockhoff and  Thanh-Do Tran, INRIA, for the fruitful discussion about MOBBOB via e-mails. 
\bibliography{mop}
\bibliographystyle{siam}

	\begin{landscape}
		\begin{figure}[htbp]
			\begin{center}
				\renewcommand{\arraystretch}{1.1}
				\resizebox{1.4\textwidth}{!}{
					\includegraphics[height=100cm]{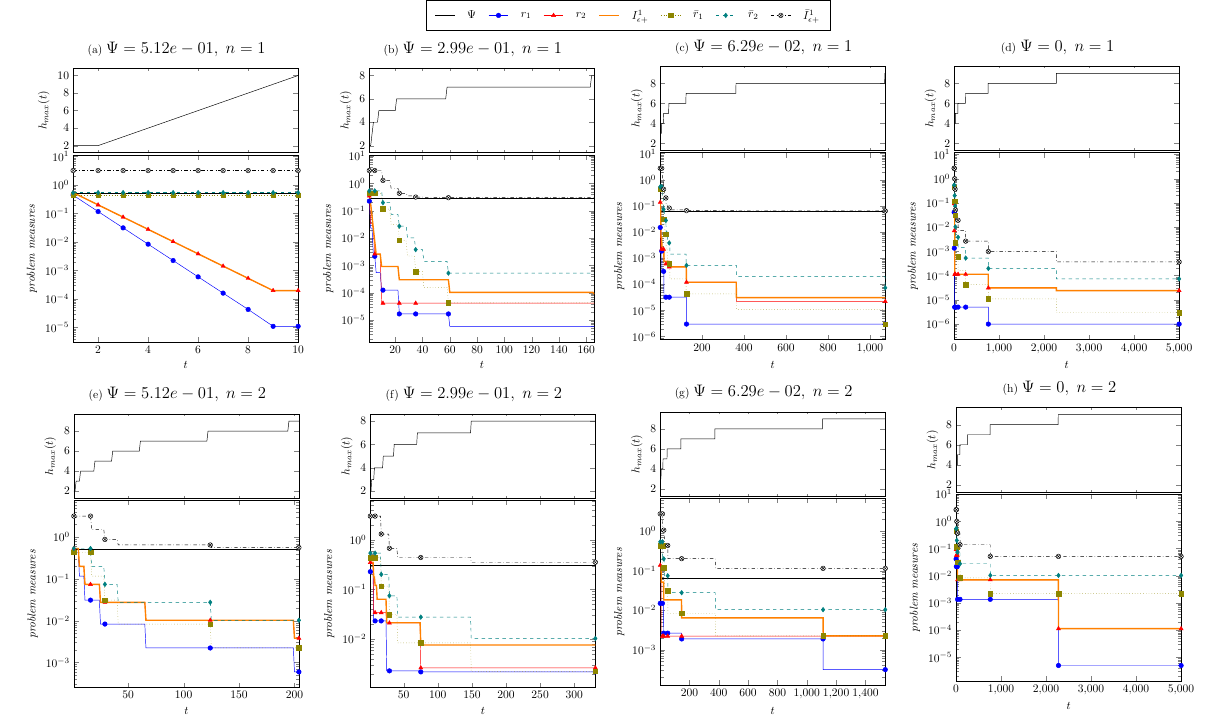}}
			\end{center}
			\caption{Empirical validation of \alg{MO-SOO}'s finite-time analysis for eight instances $\{a,\ldots,h\}$ of the bi-objective problem of Figure~\ref{fig:bi_obj_ill}. Each plot shows the problem measures, namely the loss measures~$r_1(t),\;r_2(t)$ and the indicator~\IUnaryAddInd, as well as their upper bounds (denoted by $\bar{r}_1,\;\bar{r}_2,$ and $\bar{I}^{1}_{\epsilon+}$, respectively) as a function of the number of iterations~$t$ with a computational budget of $v={10}^4$ function evaluations. The upper bounds are obtained via symbolic computation of the~\eqref{eq:loss_bound} and~\eqref{eq:indicator_bound} equations using \textsf{MATLAB}'s Symbolic Math Toolbox. The header of each instance's plot reports the decision space dimension~$n$ and the conflict dimension~$\Psi$. The $j$-optimal solutions $(x^*_1,\;x^*_2)$ are fixed as follows:  $(\myvec{0},\myvec{1})$ for $(a) \text{ and } (e)$, $(\myvec{0.21},\myvec{0.81})$ for $(b) \text{ and } (f)$, $(\myvec{0.47},\myvec{0.61})$ for $(c) \text{ and } (g)$, and $(\myvec{0.57},\myvec{0.57})$ for $(d) \text{ and } (h)$.}
			\label{fig:bounds_validation}		
		\end{figure}
	\end{landscape}

\end{document}